\documentclass[final,reqno]{amsart}
\usepackage{amsfonts,amscd,amssymb,amsmath,mathrsfs}
\usepackage{enumerate}
\usepackage[all]{xy}

\newtheorem{thm}{Theorem}[section]
\newtheorem{lem}[thm]{Lemma}
\newtheorem{prop}[thm]{Proposition}
\newtheorem{cor}[thm]{Corollary}
\theoremstyle{definition}
\newtheorem{defn}{Definition}[section]
\newtheorem{ex}[thm]{Example}
\newtheorem{rem}[thm]{Remarks}
\newtheorem{rmk}{Remark}[section]
\newtheorem{notat}[defn]{Notation}

\DeclareMathOperator{\D}{D}
\DeclareMathOperator{\End}{End}
\DeclareMathOperator{\Hom}{Hom}
\DeclareMathOperator{\Ker}{Ker}
\DeclareMathOperator{\Coker}{Coker}
\DeclareMathOperator{\Image}{Im}

\DeclareMathOperator{\sem}{\!-ss}
\DeclareMathOperator{\proj}{\!-proj}
\DeclareMathOperator{\inj}{\!-inj}
\DeclareMathOperator{\add}{add}

\DeclareMathOperator{\spaces}{\!-sp}
\DeclareMathOperator{\modules}{\!-mod}
\DeclareMathOperator{\Modules}{\!-Mod}
\DeclareMathOperator{\res}{res}
\DeclareMathOperator{\ind}{ind}
\DeclareMathOperator{\coind}{coind}
\DeclareMathOperator{\op}{^{\!op}}

\DeclareMathOperator{\Soc}{soc}

\DeclareMathOperator{\st}{\,|\,}

\begin{document}
\title[projectivization]{Adjoint Functors, Projectivization, and Differentiation Algorithms for Representations of Partially Ordered Sets}
\author{Mark Kleiner and Markus Reitenbach}
\address{Department of Mathematics, Syracuse University, Syracuse, NY 13244}
\email{mkleiner@syr.edu}
\address{Department of Mathematics and Statistics, Colorado Mesa University, Grand Junction, CO 81501}
\email{mreitenb@coloradomesa.edu}

\keywords{Partially ordered set; Representation; Differentiation algorithm; Adjoint functors;
Projectivization}
\subjclass[2000]{Primary 16G20}
\begin{abstract}
Adjoint functors and projectivization in representation theory of partially ordered sets are used to generalize the algorithms of differentiation by a maximal and by a minimal point.  Conceptual explanations are given for the combinatorial construction of the derived set and for the differentiation functor.\end{abstract}
\maketitle

\section{Introduction}\label{intro}
Throughout this paper $S$ is a finite partially ordered set (poset), and $k$ is a field.

An $S$-{\it space}  is a family $\mathbf V=\bigl(V, V(s)\bigr)_{s\in S},$ where $V$ is a finite dimensional $k$-vector space, $V(s)$ is a subspace of $V$ for each $s,$ and  $s\leq t$ in $S$ implies $V(s)\subseteq V(t).$ Here $V$ is called the {\em ambient space} of  $\mathbf V,$ and the $V(s)$ are called the {\em subspaces.}  One defines a morphism and a direct sum of $S$-spaces in a natural way.  An $S$-space is {\em indecomposable} if it is not isomorphic to a direct sum of two nonzero $S$-spaces. We denote by $S\spaces$ the category of $S$-spaces.  Introduced by Gabriel in ~\cite{G},  it is closely related to the category of {\em representations} of the poset $S$ originally introduced by Nazarova and Roiter ~\cite{NR}. Because of this close relationship, $S$-spaces are often also called representations of $S.$

The theory of $S$-spaces, and of representations of posets, has had many applications in the study of representations of finite-dimensional algebras, lattices over orders, and in other areas of mathematics ~\cite{GR,R,S}.  The first, and still most important, technical tool for the study of $S$-spaces themselves was provided by the so-called {\em differentiation algorithms}.  Given a poset $S,$ one constructs in a purely combinatorial way a {\em derived poset} $S'$  with the property that the categories $S\spaces$ and $S'\spaces$ are \lq\lq closely related,"  where the meaning of the latter depends on the problem one intends to solve.  For instance, if the problem is to determine whether $S\spaces$ is of {\em finite representation type}, i.e., has only finitely many nonisomorphic indecomposable objects, then  the two categories are closely related if there is a bijection, up to a finite number of elements, between the sets of isomorphism classes of indecomposable $S$-spaces and $S'$-spaces.  One then iterates the procedure and considers  a sequence of derived posets with the goal of obtaining an $m$-th derived poset $S^{(m)}$ for which the category $S^{(m)}\spaces$ is well understood, so that the problem at hand is easy to solve for $S^{(m)}.$  The obtained solution then also holds for the original poset $S.$  

The advantage of this method is that one replaces a difficult study of $S$-spaces with an easier study (often!) of the combinatorics of differentiations. The hard part is the justification of the algorithm, which must rely on the properties of $S$-spaces.

The aim of this paper is to generalize the algorithms of differentiation with respect to a {\em maximal element} of $S,$ introduced by Nazarova and Roiter in ~\cite{NR}, and with respect to a {\em minimal element} of $S,$ introduced by Zavadskij (see ~\cite{N}),  and to give a conceptual explanation of the two algorithms and their generalizations.  Both algorithms  involve a combinatorial construction of the derived poset $S'$ from $S,$ as well as a construction of a functor $S\spaces\to S'\spaces,$ the {\em differentiation functor,} that has nice properties.  The existing descriptions of both algorithms do not explain where the combinatorial construction comes from and present the differentiation functor as an ad hoc computational procedure.  

We introduce the algorithms of differentiation with respect to a  {\em principal filter} and to a {\em principal ideal} of $S$ that generalize those with respect to a minimal element and to a maximal element, respectively, and show for both generalizations that the differentiation functor is a composition of three functors, two of which are analogs of the {\em restriction} and {\em induction} functors from the representation theory of finite groups and the third is  a straightforward reduction of the size of the ambient space.  We also show for both algorithms that  the combinatorial construction of the derived poset $S'$ from the given poset $S$ is imposed on us by the fact that the projectivization procedure due to Auslander (see ~\cite{ARS}) is an ingredient of the differentiation functor.

In Section \ref{prel},  we review  combinatorics of posets and general properties of the category $S\spaces$ (see ~\cite{G,G1,GR}).  Unless $S=\emptyset,$ the category $S\spaces$ is not abelian but it has an {\em exact structure} based on the notion of a {\em proper} morphism, to be defined later, and it has enough (relatively) projectives and injectives.    If $R$ is a subset of a poset $S,$ the restriction functor $\res^S_R$ applied to an $S$-space preserves the ambient space and \lq\lq forgets" the subspaces associated with the elements of $S\setminus R.$  The induction functor $\ind^S_R$ is a left adjoint, and the  {\em coinduction} functor $\coind^S_R$ is a right adjoint, of $\res^S_R.$  In ~\cite{S}, the induction and coinduction are called the lower and upper induction, respectively. Recall that if $G$ is a finite group  with a subgroup $H,$  the induction functor $\ind^G_H$ is both a left and right adjoint of the restriction functor $\res^G_H.$  In addition to reviewing known facts, we present new  results about the restriction, induction, and coinduction that play a crucial role in the rest of the paper. In particular,  although the restriction functor $\res^S_R$ generally is not full, it satisfies a weaker but still useful condition   provided $R$ is either an ideal or a filter of $S;$ the condition seems interesting on its own.  In this section we also review an equivalence between the category $S\spaces$ and the category of finitely generated socle-projective modules over the incidence algebra of  the enlargement of $S$  by a unique maximal element.

Although relatively projective and relatively injective objects in various categories have been studied extensively, relatively semisimple objects seem to be less popular.  In this paper, relatively semisimple objects in the category $S\spaces$ play an important role, and we study them in Section \ref{onedim}. We say that a nonzero $S$-space $\mathbf V$ is {\em (relatively) simple} if every  nonzero proper monomorphism $\mathbf U\to \mathbf V$ in $S\spaces$ is an isomorphism,   an $S$-space is  {\em (relatively) semisimple} if it is isomorphic to a direct sum of simple $S$-spaces, and we denote by $S\sem$ the full subcategory of $S\spaces$ determined by the semisimple $S$-spaces;  simple and semisimple $S$-spaces are called sp-simple and sp-semisimple, respectively, in  ~\cite{S}.  We recall that  an $S$-space is simple if and only if its ambient space is one-dimensional, and that there is a bijection between the set of isomorphism  classes of simple $S$-spaces and the set ${\mathcal A}(S)$  of {\em antichains} of $S,$ where an antichain is a subset of $S$ that contains no two distinct comparable elements.  If $A,B\in \mathcal A(S)$ and $k_A,k_B$ are representatives of the corresponding isomorphism classes of simple $S$-spaces, we write $A\le B$ if there exists a nonzero  morphism $k_B\to k_A,$  which turns  $\mathcal A(S)$ into a poset that contains $S$ and whose unique maximal element is the empty antichain; we denote the poset by $\hat{\mathcal A}(S).$  Let $\mathcal U$ be the direct sum of a complete set of representatives of the isomorphism classes of simple $S$-spaces.  We prove that the incidence algebra $k\hat{\mathcal A}(S)$ is the opposite of the endomorphism ring of $\mathcal U.$  

Section \ref{pr} deals with projectivization.  Since $\mathcal U$ is an additive generator of $S\sem,$ the representable functor  $\Hom_{S\spaces}(\mathcal U,\underline{\phantom{x}})$ induces an equivalence between the category $S\sem$ and the category of finitely generated projective $k\hat{\mathcal A}(S)$-modules.  Denote by $\hat a(S)$ the set of nonempty antichains of $S.$ Then  $\hat{\mathcal A}(S)$ is the enlargement of $\hat a(S)$ by a unique maximal element, and the category  of socle-projective $k\hat{\mathcal A}(S)$-modules is equivalent to the category $\hat a(S)\spaces.$  Composing the two equivalences, we obtain that the category $S\sem$ is equivalent to the category of projective $\hat a(S)$-spaces.   In particular, if the {\em width} of $S,\ w(S),$ does not exceed two, where the width of a poset is the largest possible cardinality of an antichain in it, the category $S\spaces$ is equivalent to the category of projective $\hat a(S)$-spaces because $S\spaces=S\sem$ if and only if $w(S)\le2.$  The latter equivalence is an analog of the following well-known fact about representations of algebras.  If $\Lambda$ is an artin algebra of finite representation type and $\Gamma$ is its Auslander algebra, the category of finitely generated $\Lambda$-modules is equivalent to the category of finitely generated projective $\Gamma$-modules.   We finish the section by proving for categories of $S$-spaces a more general version of the above equivalence.  If $T$ is a subposet of width $\le2$ of a poset $S$  but no assumption on $w(S)$ is made,  for a suitable poset $P$ the functor $\coind^P_S$ induces an equivalence between the category $S\spaces$ and the full subcategory of $P\spaces$ determined by the $P$-spaces $\mathbf V$ for which $\res^P_{\hat a(T)}\mathbf V$ is a projective $\hat a(T)$-space.  This shows that the poset $\hat a(T),$ which is  the main ingredient of the combinatorial construction of the derived poset $S'$ from $S,$ comes from projectivization.  Using the contravariant representable functor $\Hom_{S\spaces}(\underline{\phantom{x}},\mathcal U),$ we prove that for a suitable poset $Q$ the functor $\ind^Q_S$  induces an equivalence between the category $S\spaces$ and the full subcategory of $Q\spaces$ determined by the $Q$-spaces $\mathbf V$ for which $\res^Q_{\check a(T)}\mathbf V$ is an injective $\check a(T)$-space.  Here $\check a(T)$ is the set of nonempty antichains of $T$ with a partial order different from that of $\hat a(T).$

Finally, Section 5 presents the construction and justification of the differentiation algorithms with respect to a principal filter and to a principal ideal.  It begins with a description of a functor that we characterized earlier as a straightforward reduction of the size of the ambient space. Let  $\mathbf V=\bigl(V, V(s)\bigr)_{s\in S}$  be an arbitrary $S$-space. For any  subspace $U$ of the ambient space $V,$ one can construct two  $S$-spaces in an obvious way: one with the ambient space $U$ and subspaces $V(s)\cap U,\ s\in S,$  the other with the ambient space $V/U$ and subspaces $\bigl(V(s)+U\bigr)/U,\ s\in S.$  If $p\in S$ is fixed and $U=V(p),$   both constructions become functorial in $\mathbf V,$ thus giving rise to two endofunctors of $S\spaces,$ $E^p$ and $E_p,$ respectively. 

For any $p\in S,$ the subset $\langle p\rangle=\{s\in S\,|\,p\le s\}$ is called the {\em principal filter} of $S$ generated by $p.$   We set $S_{\langle p\rangle}=\langle p\rangle\cup \hat a\bigl(S\setminus \langle p\rangle\bigr)$ and $S_p=S_{\langle p\rangle}\setminus (p)$  where $(p)=\{t\in S_{\langle p\rangle}\,|\,t\le p\}$ is the {\em principal ideal} of $S_{\langle p\rangle}$ generated by $p.$  When $w(S\setminus \langle p\rangle)\le2,$ we construct the differentiation functor $\res^{S_{\langle p\rangle}}_{S_p}E_p\coind^{S_{\langle p\rangle}}_S:S\spaces\to S_p\spaces$  with respect to the principal filter $\langle p\rangle.$  If $p$ is a minimal element of $S,$ our formula agrees with the known one.  Similarly, we construct a differentiation functor with respect to the principal ideal of $S$ generated by $p.$  The proofs use properties of  adjoint functors specialized to restriction, induction, and coinduction, as well as the existence and properties of projective covers and injective envelopes in $S\spaces.$ 

We end the introduction by reminding the reader that many authors have studied and applied, and continue to study, apply, and generalize, various differentiation algorithms for representations of posets.  For example, Zavadskij introduced the two-point differentiation algorithm ~\cite{Z1} motivated by his joint work with Kirichenko ~\cite{ZKi} on integral representation theory, and he recently came up with a kind of differentiation algorithm defined in a rather general context of poset representations ~\cite{Z2}; Bautista and Simson ~\cite{BS} obtained a generalized differentiation algorithm for a certain class of rings; Arnold in ~\cite{A} and jointly with Simson in ~\cite{AS} showed that poset representations and differentiation algorithms have useful applications in the study of subcategories of the category of abelian groups; Rump ~\cite{Ru1,Ru2,Ru3} extended differentiation algorithms to lattices over orders and modules over artinian rings, etc.

The authors are grateful to the referees for their useful comments.

\section{Preliminaries}\label{prel} 

\subsection{Filters, ideals, and antichains} We recall several facts, mostly without proof, needed in the sequel; see ~\cite{E}.

Throughout the paper,  $S^{\mathrm{op}}=(S,\underset{\mathrm{op}}\le)$ stands for the opposite poset of $S=(S,\le)$  where $a\underset{\mathrm{op}}\le b$ if and only if $b\le a$,  for all $a,b\in S$.  If $T$ is a subset of $S$, we always view it as a poset $T=(T,\le)$ with respect to the same partial order. 

A subset $F$ of $S$ is a {\it  filter} if, for all $s\in S$, we have that $t\in F$ and $t\le s$ imply $s\in F.$ A subset $I$ of $S$ is an {\it ideal}   if, for all $s\in S$, we have that $t\in I$ and $s\le t$ imply $s\in I.$   Of course, $F$ is a filter if and only if $S\setminus F$ is an ideal. If $T$ is a subset of $S,$ then $\langle T\rangle$  is the filter {\em generated} by $T,$ that is the intersection of all filters of $S$ containing $T.$  The ideal $(T)$ generated by $T$ is defined similarly.  If $p\in S,$ then $\langle p\rangle $ (respectively, $(p)$) denotes the {\em principal filter} (respectively, {\em principal ideal}) of $S$ generated by $p$, i.e.,  $\langle p\rangle=\{s\in S\,|\,p\le s\}$ and $(p)=\{s\in S\,|\,s\le p\}.$   We denote by $\mathcal F(S)$ the set of filters of $S$, and by $\mathcal I(S)$ the set of ideals of $S$. 

We will later use  the following easily verifiable statement.
\begin{prop}\label{prel.5}  Let $S$ be a poset with a subset $T$.
 \begin{itemize}
\item[(a)]  If $F$ is a filter of $S$, then $ F\cap T$ is a filter of $T$.
\item[(b)] If $I$ is an ideal of $S$, then $I\cap T$ is an ideal of $T$. 
\end{itemize}
\end{prop}

Recall that a subset $A$ of $S$ is an {\it antichain} if no two distinct elements of $A$ are comparable.  
We denote by $\mathcal A(S)$ the set of antichains, and by $a(S)$ the set of nonempty antichains, of $S$  so that ${\mathcal A}(S)= a(S)\cup\{\emptyset\}.$  The {\em width} of $S,\ w(S),$ is the largest possible cardinality of an  antichain in $S.$  For all subsets $T$ of $S$, $\min T$ (respectively, $\max T$) denotes the set of minimal (respectively, maximal) elements of $T;$ clearly, $\min T,\,\max T\in\mathcal A(S).$ For $a,b\in S$, an element $a\wedge b\in S$ is the {\it meet} of $a$ and $b$ if, for all $s\in S$, $s\le a$ and $s\le b$ imply $s\le a\wedge b.$ An element  $a\vee b\in S$ is the {\it join} of $a$ and $b$ if, for all $s\in S$, $a\le s$ and $b\le s$ imply $a\vee b\le s$.  A poset $S$ is a {\it meet-semilattice}  ({\it join-semilattice}) if the meet (join) exists for every two elements of $S.$

The following two propositions  relating the sets $\mathcal A(S)$, $\mathcal F(S)$, and $\mathcal I(S)$ are well known, and we present them without proof.

\begin{prop}\label{prel1}
Let $S$ be a poset.
\begin{itemize}
\item[(a)] The functions $\mathcal F(S)\to\mathcal A(S)$ given by $F\mapsto\min F$ for all $F\in\mathcal F(S)$, and $\mathcal A(S)\to\mathcal F(S)$ given by $A\mapsto\langle A\rangle$ for all $A\in\mathcal A(S)$ are mutually inverse bijections.
\item[(b)] Let $F,G\in\mathcal F(S)$ and let $\min F=\{a_1,\dots,a_m\}$, $\min G=\{b_1,\dots,b_n\},\ m,n\ge0$. Then $F\supseteq G$ if and only if for all $j$, there exists an $i$ satisfying $a_i\le b_j$.
\item[(c)] For all $\{a_1,\dots,a_m\},\{b_1,\dots,b_n\}$ in $\mathcal A(S)$, set $\{a_1,\dots,a_m\}\le\{b_1,\dots,b_n\}$ if and only if for all $j$, there exists an $i$ satisfying $a_i\le b_j$.  Then:
\begin{itemize} 
\item [(i)] $(\mathcal A(S),\le)$ is a meet-semilattice where $A\wedge B=\min\{A\cup B\}$ for all $A, B\in\mathcal A(S)$.
\item[(ii)] $\emptyset$ is a unique maximal element and $\min S$ is a unique minimal element  of $\mathcal A(S)$.
\item[(iii)] For $s,t\in S$ we have $\{s\}\le\{t\}$ in $\mathcal A(S)$ if and only if $s\le t$ in $S$.
\item[(iv)] For $m>0$, $\{a_1,\dots,a_m\}$ is the meet of $\{a_1\},\dots,\{a_m\}$.
\end{itemize}
\end{itemize}
\end{prop}

\begin{notat}\label{prel2}  We denote by $\hat{\mathcal A}(S)$  the meet-semilattice of part (i) of Proposition \ref{prel1}(c); denote by $\hat a(S)$ the subposet of nonempty antichains in $\hat{\mathcal A}(S);$ write $s$ instead of $\{s\}$ for $s\in S$ so that $S$ becomes a subposet of $\hat a(S)$ as justified by part (iii) of Proposition \ref{prel1}(c); and  write $a_1\wedge\dots\wedge a_m$ instead of $\{a_1,\dots,a_m\}$ as justified by parts (iii) and (iv) of Proposition \ref{prel1}(c).
\end{notat}

\begin{prop}\label{prel3}
Let $S$ be a poset.
\begin{itemize}
\item[(a)] The functions $\mathcal I(S)\to\mathcal A(S)$ given by $I\mapsto\max I$ for all $I\in\mathcal I(S)$, and $\mathcal A(S)\to\mathcal I(S)$ given by $A\mapsto(A)$ for all $A\in\mathcal A(S)$ are mutually inverse bijections.
\item[(b)] Let $I,J\in\mathcal I(S)$ and let $\max I=\{a_1,\dots,a_m\}$, $\max J=\{b_1,\dots,b_n\},\ m,n\ge0$. Then $I\subseteq J$ if and only if for all $i$, there exists a $j$ satisfying $a_i\le b_j$.
\item[(c)] For all $\{a_1,\dots,a_m\},\{b_1,\dots,b_n\}$ in $\mathcal A(S)$, set $\{a_1,\dots,a_m\}\le\{b_1,\dots,b_n\}$ if and only if for all $i$, there exists a $j$ satisfying $a_i\le b_j$.  Then: 
\begin{itemize} 
\item [(i)] $(\mathcal A(S),\le)$ is a join-semilattice where $A\vee B=\max\{A\cup B\}$ for all $A, B\in\mathcal A(S)$.
\item[(ii)] $\emptyset$ is a unique minimal element and $\max S$ is a unique maximal element  of $\mathcal A(S)$.
\item[(iii)] For $s,t\in S$ we have $\{s\}\le\{t\}$ in $\mathcal A(S)$ if and only if $s\le t$ in $S$.
\item[(iv)] For $m>0$, $\{a_1,\dots,a_m\}$ is the join of $\{a_1\},\dots,\{a_m\}$.
\end{itemize}
\end{itemize}
\end{prop}

\begin{notat}\label{prel4}  We denote by $\check{\mathcal A}(S)$ the join-semilattice of part (i) of Proposition \ref{prel3}(c); denote by $\check{a}(S)$ the subposet of nonempty antichains in $\check{\mathcal A}(S);$ write $s$ instead of $\{s\}$ for $s\in S$ so that $S$ becomes a subposet of $\check{a}(S)$ as justified by part (iii) of Proposition \ref{prel3}(c); and   write $a_1\vee\dots\vee a_m$ instead of $\{a_1,\dots,a_m\}$ as justified by parts (iii) and (iv) of Proposition \ref{prel3}(c).
\end{notat}

We note that $\check{\mathcal A}(S)=\hat {\mathcal A}(S^{\mathrm{op}}).$\vskip.05in

The following two statements are important for our treatment of the differentiation algorithms.
\begin{prop}\label{prel2.5} Let  $p\in S$.
\begin{itemize}
\item[(a)]  $\hat{a}\bigl((p)\bigr)$ is the principal ideal of $\hat{a}(S)$ generated by $p$.
\item[(b)] $\check{a}\bigl(\langle p\rangle\bigr)$ is the principal filter of $\check{a}(S)$ generated by $p$.
\end{itemize}
\end{prop}
\begin{proof} (a)  Let $I$ be the principal ideal of $\hat{a}(S)$ generated by $p$ and let $A=a_1\wedge\dots\wedge a_m$ $(m>0).$  Then $A\in \hat{a}\bigl((p)\bigr)$ if and only if $a_i\in(p)$ for all $i;$ if and only if $a_i\leq p$ for all $i;$ if and only if $A\leq p;$ if and only if  $A\in I$. 

(b)  The proof is dual to that of (a).
\end{proof}

For a  subset $R$ of a poset $S,$  consider  the sets $S_R=R\,\cup\,\hat a(S\setminus R)$ and $S^R=R\cup\check a(S\setminus R)$ that are subposets of $\hat{a}(S)$ and $\check{a}(S)$, respectively. 
\begin{prop}\label{prel5.5}Let $S$ be a poset with   a subset $R$. 
 \begin{itemize}  
\item[(a)] If $R$ is a filter of $S$, then $R$ is a filter of $S_R$.
\item[(b)]  If $R$ is an ideal of $S$, then $R$ is an ideal of $S^R$.
\end{itemize}
\end{prop}
\begin{proof} (a) Let $x\in R$ and $ y\in S_R$ satisfy $x\le y$.  We have to show that $y\in R$. If  $y\in S$, this holds by assumption. If $y\not\in S$, then $y=x_1\wedge\dots\wedge x_n,\ n>1$, where $x_j\in S\setminus R$ for all $j.$  By Proposition \ref{prel1}(b), $x\le y $ implies $x\le x_j$ for all $j$.  Since $R$ is a filter of $S$, then $x_j\in R$, a contradiction.  

(b) The proof is similar to that of (a).
\end{proof}

\subsection{The category $S\spaces$ and socle-projective  modules}\label{category}  We recall several well-known facts.  For unexplained definitions and omitted proofs, see ~\cite{G, GR, Mac, R, Re, S}.

 Recall that   a {\it morphism} $f:\mathbf U\to\mathbf V$  of  $S$-spaces $\mathbf U=\bigl(U, U(s)\bigr)_{s\in S}$ and $\mathbf V=\bigl(V, V(s)\bigr)_{s\in S}$ is a $k$-linear map $f:U\to V$ satisfying $f(U(s))\subseteq V(s),\, s\in S$. The direct sum $\mathbf U\oplus \mathbf V$ is the family $\bigl(X, X(s)\bigr)_{s\in S}$ where $X=U\oplus V$ and $X(s)=U(s)\oplus V(s),\, s\in S.$

\begin{prop}\label{cat1}  Let $f:\mathbf U\to\mathbf V$ be a morphism of $S$-spaces given by a $k$-linear map $f:U\to V.$  
\begin{itemize}
\item[(a)]   $f:\mathbf U\to\mathbf V$ is a monomorphism if and only if the linear map $f:U\to V$  is injective.

\item[(b)]   $f:\mathbf U\to\mathbf V$ is an epimorphism if and only if the linear map $f:U\to V$  is surjective.

\item[(c)] The family $\mathbf X=\bigl(X, X(s)\bigr)_{s\in S}$ where $X=\Ker f$ and $X(s)=U(s)\cap\Ker f,\,s\in S,$ is an $S$-space.  The  inclusion $\varkappa:X\to U$ gives a kernel $\varkappa:\mathbf X\to\mathbf U$ of $f:\mathbf U\to\mathbf V.$
\item[(d)] The family $\mathbf Y=\bigl(Y, Y(s)\bigr)_{s\in S}$ where $Y=V/f(U)$ and $Y(s)=\bigl(V(s)+f(U)\bigr)/f(U),\,s\in S,$ is an $S$-space.  The  projection $\sigma:V\to Y$ gives a cokernel $\sigma:\mathbf V\to\mathbf Y$ of $f:\mathbf U\to\mathbf V.$
\end{itemize}
\end{prop}

The preceding proposition implies that $S\spaces$ is a Krull-Schmidt category, that is, an additive $k$-category in which idempotents split and the isomorphism ring of each indecomposable object is local.  Hence each $S$-space decomposes uniquely up to isomorphism as a  direct sum of indecomposable $S$-spaces. 

\begin{defn}\label{cat2}   A morphism $f:\mathbf U\to\mathbf V$ of $S$-spaces is said to be {\it proper} if, for all $s\in S,$ we have $f(U(s))=V(s)\cap f(U).$
\end{defn}

The following statement is straightforward.

\begin{prop}\label{cat3}  Let $\mathbf V=\bigl(V,V(s)\bigr)_{s\in S}$ be an $S$-space, and let $U$ be a subspace of the ambient space $V.$
\begin{itemize}

\item[(a)] The family $\mathbf U=\bigl(U,U(s)\bigr)_{s\in S},$ where $U(s)=V(s)\cap U,$ is the unique $S$-space with the ambient space $U$ for which the inclusion $U\hookrightarrow V$ gives a proper monomorphism $\mathbf U\to\mathbf V.$

\item[(b)] The family $\mathbf W=\bigl(W,W(s)\bigr)_{s\in S},$ where $W=V/U$ and $W(s)=\bigl(V(s)+U\bigr)/U,$ is the unique $S$-space with the ambient space $V/U$ for which the projection $V\to V/U$ gives a proper epimorphism $\mathbf V\to\mathbf W.$

\item[(c)] A kernel of  a morphism of $S$-spaces  is a proper monomorphism, and a cokernel  is a proper epimorphism.

\item[(d)]  A proper monomorphism is a kernel of its cokernel.  A proper epimorphism is a cokernel of its kernel. 
\end{itemize}
\end{prop}

For any associative ring  $\Lambda$ with unity, we denote  by $\Lambda\Modules$ (respectively, $\Lambda\modules$)  the category of left  (respectively,  finitely generated left) $\Lambda$-modules, and $\Lambda\proj$ stands for the full subcategory of $\Lambda\modules$  determined by the projective modules. In the sequel we will need an interpretation of the category $S\spaces$ as a full subcategory of the category $\Lambda\modules,$ for some finite dimensional associative $k$-algebra $\Lambda$ with unity.

Given a  finite poset $P,$ denote by $M_P(k)$  the full matrix algebra over $k$ whose rows and columns are indexed by the elements of $P.$ We write $e_{xy},\, x,y\in P,$ for the matrix unit with $1$ in row $x$ and column $y.$ The $k$-subspace of $M_P(k)$ with basis   $\{e_{ba}\,|\,a\le b,\ a,b\in P\}$  is a $k$-subalgebra  called the {\it incidence algebra} $kP$ of the poset $P$ over $k.$ The subset $\{e_{aa}\,|\,\,a\in P\}$ of the  basis is a complete set of primitive orthogonal idempotents of $kP.$  

\begin{rmk}\label{cat3.5}The set $\{e_{ab}\,|\,a\le b,\ a,b\in P\}$ is a basis for the incidence algebra $kP^{\mathrm {op}}$ of the opposite poset $P^{\mathrm {op}},$ and the  map $M_P(k)\to M_P(k)$ sending each matrix $A$ to its transpose $A^t$ induces an antiisomorphism of $k$-algebras $kP\to kP^{\mathrm {op}}.$
\end{rmk}

\begin{defn}\label{cat6} Denote by $S^\omega=S\cup\{\omega\}$ the  poset whose structure  is defined by letting the elements of $S$ retain their original partial order and setting $s<\omega,\, s\in S.$  The indecomposable module $kS^\omega e_{\omega\omega}\in kS^\omega\proj$ is one-dimensional, hence, simple. It  is a unique up to isomorphism simple  projective $kS^\omega$-module, and we denote by $kS^\omega\spaces$ the full subcategory of $kS^\omega\modules$ determined by the {\em socle-projective} modules: $M\in kS^\omega\spaces$ if and only if the socle of $M,$ $\Soc M,$ is in $kS^\omega\proj.$ 

Consider the following map $\Phi=\Phi_S:kS^\omega\modules\to S\spaces.$ For all $M\in kS^\omega\modules$, set   $\Phi M=\bigl(V,V(s)\bigr)_{s\in S}$  where $V=e_{\omega\omega}M$ and $V(s)=e_{\omega s}M.$  For all morphisms $f:M\to N$ in $kS^\omega\modules,$  set $\Phi f$ to be the restriction of  $f$ to $e_{\omega\omega}M,$ i.e.,   $\Phi f=f|e_{\omega\omega}M.$

Consider also the following map $\mathrm{\Psi}=\Psi_S:S\spaces\to kS^\omega\modules.$  For all $S$-spaces $\mathbf V=\bigl(V,V(s)\bigr)_{s\in S},$ set $\mathrm \Psi \mathbf V=\bigoplus_{t\in S^\omega}V(t)$ where $V(\omega)=V$ and the multiplication by the basis element $e_{ba}$ of $kS^\omega$ on  $ \Psi \mathbf V$ is the $k$-linear operator that induces the embedding $V(a)\hookrightarrow V(b)$ on $V(a)$ and sends the other direct summands to $0$.  For all morphisms $f:\mathbf V\to\mathbf W$ in $S\spaces,$  set  $\mathrm \Psi f=\bigoplus_{t\in S^\omega}f|{V(t)}.$  
\end{defn}

Note that $kS^\omega\spaces$ contains $kS^\omega\proj$, since $\Soc kS^\omega e_{ss}\cong kS^\omega e_{\omega\omega}$ for $s\in S$. 

\begin{prop}\label{cat7} 
\begin{itemize}
\item[(a)]  $\Phi:kS^\omega\modules\to S\spaces$ and $\mathrm{\Psi}:S\spaces\to kS^\omega\modules$ are $k$-linear functors.

\item[(b)] The image of   $\mathrm{\Psi}$ is contained in $kS^\omega\spaces,$ and  $\Psi:S\spaces\to kS^\omega\spaces$ is a dense functor.

\item[(c)] We have $\Phi\Psi=1_{S\spaces}.$  In particular, $\Psi:S\spaces\to kS^\omega\spaces$ is an equivalence of categories.
\end{itemize}
\end{prop}

The functor $\Phi$ is  called an adjustment functor in ~\cite[p. 190]{S}.

\begin{defn}\label{cat5} An $S$-space $\mathbf P$ is called {\em (relatively) projective} if for every proper epimorphism $f:\mathbf U\to\mathbf V$ and every morphism $h:\mathbf P\to\mathbf V$ of $S$-spaces there exists a morphism $g:\mathbf P\to\mathbf U$ satisfying $h=fg.$ We denote by $S\proj$ the full subcategory of $S\spaces$ determined by the projective $S$-spaces. A morphism $f:\mathbf U\to\mathbf V$ is called {\em right minimal} if every morphism $g:\mathbf U\to\mathbf U$ satisfying $f=fg$ is an automorphism. An epimorphism $f:\mathbf U\to\mathbf V$ is called an {\it essential epimorphism} if for every morphism $g:\mathbf X\to\mathbf U$, $g$ is a proper epimorphism if and only if $fg$ is a proper epimorphism. A {\it projective cover} of an $S$-space $\mathbf V$  is an essential epimorphism $f:\mathbf P\to\mathbf V$ with $\mathbf P$ projective. {\em Injectives, left minimal morphisms, essential monomorphisms, and injective envelopes} are defined in a similar way, and we denote by $S\inj$ the  full subcategory of $S\spaces$ determined by the injective $S$-spaces.
\end{defn} 

Since $\mathrm \Psi$ is not dense, it is a right inverse  but not an inverse of $\Phi$.  The next statement says in  particular that the restrictions of $\mathrm \Psi$ and $\Phi$ to the full subcategories of projective objects \emph{are} inverses of each other.

\begin{thm}\label{projectives}
\begin{itemize}
\item[(a)] An $S$-space $\mathbf P$ is projective if and only if $\mathrm \Psi\mathbf P$ is a projective $kS^\omega$-module.
\item[(b)] Every indecomposable projective $S$-space is isomorphic to one, and only one, of the spaces $\mathbf P_t$, $t\in S^\omega$, where $\mathbf P_t=\bigl(P_t,P_t(s)\bigr)_{s\in S}$ with $P_t=k$ and $P_t(s)=
\begin{cases}
    k   &   \text{if $s\geq t$},\\
    0   &   \text{otherwise.}
\end{cases}$
\item[(c)] Every projective $S$-space is isomorphic to $\bigoplus_{t\in S^\omega}\mathbf P_t^{n_t},$ for unique integers $n_t\geq 0$.
\item[(d)] The functors $\mathrm \Psi$ and $\Phi$ induce mutually inverse equivalences of categories
$$
\xymatrix{S\proj\ar@/^/[rr]^{\mathrm \Psi}\ && kS^\omega\proj\ar@/^/[ll]^\Phi}.
$$
\item[(e)] Every $S$-space has a projective cover.
\item[(f)]  A proper epimorphism $f:\mathbf P\to\mathbf V$ with $\mathbf P\in S\proj$ is a projective cover if and only if the morphism $f$ is right minimal.
\end{itemize}
\end{thm}

\begin{defn}\label{cat4} The  vector space duality $\D=\Hom_k(-,k)$ extends to a  {\em duality} $\D:S\spaces\to S\op\spaces$ as follows.  For each $S$-space $\mathbf V=\bigl(V,V(s)\bigr)_{s\in S}$, set $\D\mathbf V=\bigl(X,X(s)\bigr)_{s\in S}$ where $X=\D V$ and  $X(s)=V(s)^\bot=\{g\in\D V\st g(V(s))=0\},\,s\in S.$ For each morphism $f:\mathbf U\to\mathbf V,$ where  $\mathbf U=\bigl(U,U(s)\bigr)_{s\in S},$ the morphism $\D f:\D\mathbf V\to\D\mathbf U$ is given by the $k$-linear map $\D f:\D V\to\D U.$  By restriction one obtains dualities  $\D:S\proj\to S\op\inj$ and $\D:S\inj\to S\op\proj$. 
\end{defn}

\begin{prop}\label{cat10}  If $f:\mathbf U\to\mathbf V$ is a proper morphism of $S$-spaces, then $\D f:\D\mathbf V\to\D\mathbf U$ is a proper morphism of $S\op$-spaces.
\end{prop}

Applying the duality $\D$ of Definition \ref{cat4} to Theorem \ref{projectives} and using   Proposition \ref{cat10}, one gets the following description of the category $S\inj.$

\begin{defn}\label{cat9} Denote by $S_0=S\cup\{0\}$ the  poset whose structure is defined by letting the elements of $S$ retain their original partial order and setting $0<s,\, s\in S.$ 
\end{defn} 

\begin{thm}\label{injectives}
\begin{itemize}
\item[(a)] Every indecomposable injective $S$-space is isomorphic to one, and only one, of the spaces $\mathbf I_t,\,t\in S_0,$ where $\mathbf I_t=\bigl(I_t,I_t(s)\bigr)_{s\in S}$ with $I_t=k$ and $I_t(s)=
\begin{cases}
    0   &   \text{if $s\leq t$},\\
    k   &   \text{otherwise.}
\end{cases}$
\item[(b)] Every injective $S$-space is isomorphic to $\bigoplus_{t\in S_0}\mathbf I_t^{n_t},$ for unique integers $n_t\geq 0$.
\item[(c)] Every $S$-space has an injective envelope.
\item[(d)]  A proper monomorphism $g:\mathbf V\to\mathbf I$ with $\mathbf I\in S\inj$ is an injective envelope if and only if the morphism $g$ is left minimal.
\end{itemize}
\end{thm}
                                                                                    
\subsection{Subposets and  adjoint functors}\label{adj} We recall known results and prove facts needed in the sequel.  For unexplained terminology, see ~\cite{M}.

If $R$ is a subset of a poset $S$,   the {\it restriction} (forgetful) functor $\res^S_R:S\spaces\to R\spaces$ sends an  $S$-space $\mathbf V=(V,V(s))_{s\in S}$ to the $R$-space $\res^S_R\mathbf V=(V,V(r))_{r\in R}$, and it sends a morphism $f:\mathbf V\to\mathbf W$ in $S\spaces$ given by a $k$-linear map $f:V\to W$, where  $\mathbf W=(W,W(s))_{s\in S}$, to the morphism $\res^S_Rf:\res^S_R\mathbf V\to\res^S_R\mathbf W$ in $R\spaces$ given by the same linear map $f:V\to W$.  

\begin{rmk}\label{adj2} The restriction functor $\res^S_R:S\spaces\to R\spaces$ preserves proper morphisms.
\end{rmk}

The functor  $\res^S_R$ has a left adjoint and a right adjoint. Denote by $\ind^S_R:R\spaces\to S\spaces$ the following functor.  For an $R$-space $\mathbf V=(V,V(r))_{r\in R}$, the $S$-space $\ind^S_R\mathbf V=(X,X(s))_{s\in S}$ is given by $X=V$ and $X(s)=\underset{r\in R,r\le s}\sum V(r)$ (if no $r\in R$ satisfies $r\le s$ then $X(s)=0$).  For a morphism $f:\mathbf V\to\mathbf W$ in $R\spaces$ given by a $k$-linear map $f:V\to W$, $\ind^S_Rf:\ind^S_R\mathbf V\to\ind^S_R\mathbf W$ is the morphism in $S\spaces$ given by the same linear map $f:V\to W$.  We call  $\ind^S_R$ the {\it induction} functor.

Denote by $\coind^S_R:R\spaces\to S\spaces$ the following functor.  For an $R$-space $\mathbf V=(V,V(r))_{r\in R}$, the $S$-space $\coind^S_R\mathbf V=(X,X(s))_{s\in S}$ is given by $X=V$ and $X(s)=\underset{r\in R,s\le r}\cap V(r)$ (if no $r\in R$ satisfies $s\le r$ then $X(s)=V$).  For a morphism $f:\mathbf V\to\mathbf W$ in $R\spaces$ given by a $k$-linear map $f:V\to W$, $\coind^S_Rf:\coind^S_R\mathbf V\to\coind^S_R\mathbf W$ is the morphism in $S\spaces$ given by the same linear map $f:V\to W$.  We call  $\coind^S_R$ the {\it coinduction} functor.

\begin{defn}\label{adj3}  An $S$-space $\mathbf V=(V,V(s))_{s\in S}$ is  {\em trivial} at $t\in S$ if $V(t)=0,$ and it is  {\em full} at $t$ if $V(t)=V.$  If $R$ is a subset of $S,$ then $\mathbf V$ is trivial (full) at $R$ if, for all $r\in R,$ $\mathbf V$ is trivial (full) at $r.$
\end{defn}
 
\begin{rmk}\label{adj4} For any $S$-space $\mathbf V,$ the set of elements of $S$ at which $\mathbf V$ is trivial is an ideal of $S,$ and the set of elements of $S$ at which $\mathbf V$ is full is a filter of $S.$
\end{rmk}

For future reference, the following two propositions record several easily verifiable facts (see ~\cite[Propositions 5.14 and 5.16, Exercise 5.24]{S}).

\begin{prop}\label{prel6}  Let $R$ be a subset of a poset $S$ and let $\mathbf V\in R\spaces,\mathbf W\in S\spaces$.
\begin{itemize}

\item[(a)] There exist isomorphisms of $k$-spaces 
$$\Hom_S(\ind^S_R\mathbf V,\mathbf W)\cong\Hom_R(\mathbf V,\res^S_R\mathbf W)$$  and $$\Hom_R(\res^S_R\mathbf W,\mathbf V)\cong\Hom_S(\mathbf W,\coind^S_R\mathbf V)$$ functorial in $\mathbf V$ and $\mathbf W.$  In other words, $\ind^S_R$  is a left adjoint of $\res^S_R$, and $\coind^S_R$  is a right adjoint of $\res^S_R$.

\item[(b)] $\res^S_R$ is a faithful additive functor, and $\ind^S_R$ and $\coind^S_R$ are fully faithful additive functors satisfying 
$$\res^S_R\ind^S_R=\res^S_R\coind^S_R=1_{R\spaces}.$$  
In particular, $\res^S_R$ is a dense functor, and both $\ind^S_R$ and $\coind^S_R$ reflect isomorphisms.

\item[(c)] If $R\subseteq T\subseteq S$ then 
$$\res^T_R\res^S_T=\res^S_R\mathrm{\ and\ }\res^S_R\ind^S_T=\res^S_R\coind^S_T=\res^T_R.$$

\item[(d)] Let $U\subseteq S$ and $S=R\cup U$.  If $\mathbf X,\mathbf Y\in S\spaces$ then $\mathbf X=\mathbf Y$ if and only if $\res^S_R\mathbf X=\res^S_R\mathbf Y$ and $\res^S_U\mathbf X=\res^S_U\mathbf Y$.

\item[(e)] Let $R$ be a filter of $S$ and let $\mathfrak C$ be the full subcategory  of $S\spaces$ determined by the $S$-spaces trivial at $S\setminus R.$  The restriction of $\res^S_R$ to $\mathfrak C,$ $\res^S_R|\mathfrak C:\mathfrak C\to R\spaces,$ is an equivalence of categories.

\item[(f)] Let $R$ be an ideal of $S$ and let $\mathfrak D$ be the full subcategory  of $S\spaces$ determined by the $S$-spaces full at $S\setminus R.$  The restriction of $\res^S_R$ to $\mathfrak D,$ $\res^S_R|\mathfrak D:\mathfrak D\to R\spaces,$ is an equivalence of categories.
\end{itemize}
\end{prop}\vskip.1in

We will use the following statement in the section on projectivization.

\begin{prop}\label{prel6.5} Let $T$ be a subset of a poset $S.$
\begin{itemize}
\item[(a)]  For the poset $S_{S\setminus T}=(S\setminus T)\cup\hat a(T)$ we have $\res^{S_{S\setminus T}}_{\hat a(T)}\coind^{S_{S\setminus T}}_{S}=\coind^{\hat a(T)}_{T}\res^S_T.$
\item[(b)]  For the poset $S^{S\setminus T}=(S\setminus T)\cup\check a(T)$ we have $\res^{S^{S\setminus T}}_{\check a(T)}\ind^{S^{S\setminus T}}_{S}=\ind^{\check a(T)}_{T}\res^S_T.$
\end{itemize}
\end{prop}
\begin{proof} (a) For an $S$-space $\mathbf V=\bigl(V,V(s)\bigr)_{s\in S},$ set $\res^{S_{S\setminus T}}_{\hat a(T)}\coind^{S_{S\setminus T}}_{S}\mathbf V=\mathbf X=\bigl(X,X(A)\bigr)_{A\in \hat a(T)}$ and $\coind^{\hat a(T)}_{T}\res^S_T\mathbf V=\mathbf Y=\bigl(Y,Y(A)\bigr)_{A\in \hat a(T)}.$  To show that the two functors in question coincide on objects, we check that $\mathbf X=\mathbf  Y.$ 

Clearly, $X=Y=V.$ For each $A=a_1\wedge\dots\wedge a_m,\ m>0,$  we have $X(A)=\underset{s\in S,A\le s}\cap V(s)$.  By Proposition \ref{prel1}(c), $A\le s$ if and only if $a_i\le s,$ for some $i.$ It follows that $A\le s$ implies $V(a_i)\subseteq V(s),$ for some $i,$ whence $X(A)=\underset{s\in S,A\le s}\bigcap V(s)=\underset{i=1}{\overset{m}\bigcap} V(a_i)$.  A similar argument shows that  $Y(A)=\underset{i=1}{\overset{m}\bigcap} V(a_i)$.
 
 It follows immediately from the definitions of restriction and coinduction  that the two functors in question coincide on morphisms.
 
 (b) The proof is dual to that of (a).
\end{proof}

\begin{prop}\label{prel7} Let $R$ be a subset of a poset $S$ and denote by the same symbol $\D$ the duality on $S\spaces$ and on $R\spaces.$ Then the following diagrams commute.\vskip.1in

\centerline{$
\begin{CD}S\spaces@>\res_R^S>>R\spaces\\
@V \D VV@VV \D V\\
S\op\spaces@>\res_{R\op}^{S\op}>>R\op\spaces
\end{CD}
$\hskip.5in
$
\begin{CD}R\spaces@>\ind_R^S>>S\spaces\\
@V \D VV@VV \D V\\
R\op\spaces@>\coind_{R\op}^{S\op}>>S\op\spaces
\end{CD}
$}
\end{prop}
\vskip.1in

According to Proposition \ref{prel6}(b), $\res^S_R$ is a faithful and dense functor but, generally speaking, not a full functor.  However,  if $R$ is either a filter or an ideal of $S,$ the functor $\res^S_R$ has properties that can be viewed as a weak version of being full.

\begin{defn}\label{adj.5}  A functor $F:\mathfrak A\to\mathfrak B$ is said to be {\em right quasi full} if for every commutative diagram
$$ 
\begin{CD}Fu@>f>>Fv\\
@V \alpha VV@VV F\beta V\\
Fy@>g>>Fz
\end{CD}
$$
in $\mathfrak B$ there exist morphisms $f':u'\to v,\ g':y'\to z,\ \alpha':u'\to y'$ satisfying the following two conditions.
\begin{itemize}
\item[(a)] $f=Ff',\ g=Fg',\ \alpha=F\alpha'.$
\item[(b)] The diagram
$$ 
\begin{CD}u'@>f'>>v\\
@V \alpha' VV@VV \beta V\\
y'@>g'>>z
\end{CD}
$$
commutes in $\mathfrak A.$
\end{itemize}
We leave it to the reader to give the dual definition of a {\em left quasi full} functor.
\end{defn}

\begin{rem}\label{adj.6} (a) If $F:\mathfrak A\to\mathfrak B$ is a right quasi full functor, then for every morphism $f:Fu\to Fv$ in $\mathfrak B$ there exists a morphism $f':u'\to v$ satisfying $f=Ff'.$ Indeed, we can use the first  commutative diagram of Definition \ref{adj.5} by putting $g=f,\,\alpha=1_{Fu},$ and $\beta=1_v.$ Then (a) applies.

(b) A faithful functor is right quasi full if and only if it satisfies condition (a) of Definition \ref{adj.5}.

(c) A fully faithful functor is right quasi full.
\end{rem}

\begin{prop}\label{adj1} Let $R$ be an ideal of  a poset $S.$  For an $ R$-space $\mathbf U=\bigl(U,U(r)\bigr)_{r\in R}$ and an $S$-space $\mathbf V=\bigl(V,V(s)\bigr)_{s\in S},$  let $f:\mathbf U\to \res^S_R\mathbf V$ be a morphism  in $R\spaces$ that is given by a $k$-linear map $f:U\to V.$  Set $\mathbf U_f=\bigl(X,X(s)\bigr)_{s\in S}$ where  $X=U$,  $X(r)=U(r)$ for all $r\in R$, and   $X(t)=f^{-1}\bigl(V(t)\bigr)$ for all $t\in S\setminus R.$
\begin{itemize}
\item[(a)]  $\mathbf U_f$ is an $S$-space, and the linear map $f:U\to V$ gives a morphism $\hat f:\mathbf U_f\to\mathbf V$ in $S\spaces$ satisfying  $f=\res^S_R\hat f.$  Moreover, $f$ is a proper morphism (respectively, an isomorphism) if and only if so is $\hat f.$ 

\item[(b)] Consider a commutative diagram of $R$-spaces of the form
$$ 
\begin{CD}
\mathbf U@>f>>\res^S_R\mathbf V\\
@V{\alpha}VV@VV{\res^S_R\beta}V\\
\mathbf Y@>g>>\res^S_R\mathbf Z
\end{CD}
$$
 where $\mathbf Y=\bigl(Y,Y(r)\bigr)_{r\in R}\in R\spaces,\ \mathbf Z=\bigl(Z,Z(s)\bigr)_{s\in S}\in S\spaces,$ and the  morphisms $g:\mathbf Y\to \res^S_R\mathbf Z,\ \alpha:\mathbf U\to\mathbf Y,$ and $ \beta:\mathbf V\to\mathbf Z$ are given by  $k$-linear maps $g:Y\to Z,\ \alpha:U\to Y,$ and $\beta:V\to Z,$  respectively. The linear map $\alpha:U\to Y$ gives a morphism $\alpha':\mathbf U_f\to\mathbf Y_g$ in $S\spaces$ satisfying $\alpha=\res^S_R\alpha'$ and $\beta\hat f=\hat g\alpha'.$  In particular, the functor $\res^S_R$ is right quasi full.
  
\item[(c)] In the setting of (b), suppose $\beta$ is an isomorphism. Then $\mathbf U_f=\mathbf U_\alpha$ and $\alpha'=\hat \alpha:\mathbf U_\alpha \to\mathbf Y_g.$ 

\item[(d)] In the setting of (a), $f$ is right minimal if and only if so is $\hat f.$

\item[(e)] If $F$ is a filter of $S$ and $\mathbf V\cong\ind^S_F\mathbf W$, for some $\mathbf W\in F\spaces$, then  the $S$-space $\mathbf U_f=(X,X(s))$ of (a) satisfies $X(s)\subset\Ker f,$ for all $s\in R\setminus F$, and  $X(s)=\Ker f,$ for all $s\in S\setminus [R\cup F]$.
\end{itemize}
\end{prop}

\begin{proof} (a)   We only have to check that $\mathbf U_f$ is an $S$-space. Let $t_1\le t_2$ where $t_1,t_2\in S$.  If $t_1,t_2\in R$ or $t_1,t_2\in S\setminus R$, the inclusion  $X(t_1)\subseteq X(t_2)$ is obvious.  Since $R$ is an ideal of $S$, the case $t_1\in S\setminus R, \ t_2\in R$ is impossible.  If $t_1\in R,\ t_2\in S\setminus R$, then
$$X(t_1)= U(t_1)\subseteq f^{-1}[f(U(t_1)]\subseteq f^{-1}(V(t_1))\subseteq f^{-1}(V(t_2))=X(t_2).$$ 

If $\hat f$ is proper, then $f$ is proper by Remark \ref{adj2}.  If  $f$ is proper, the linear  map $f:X\to V$ satisfies $f(X(s))=V(s)\cap f(X)$ for all $s\in S$ by construction, whence $\hat f$ is proper.  

We leave it to the reader to consider the case when either $f$ or $\hat f$ is an isomorphism. 

(b)  Since restriction is a faithful functor, Remark \ref{adj.6}(b) says that  we only have to check that $\alpha(f^{-1}(V(t)))\subseteq g^{-1}(Z(t))$ for all $t\in S\setminus R.$  Suppose $u\in U$ satisfies $f(u)\in V(t).$  Since the diagram in the statement of the lemma commutes, we get $g(\alpha(u))=\beta(f(u))\in Z(t)$ because $\beta$ is a morphism in $S\spaces.$  Hence $\alpha(u)\in g^{-1}(Z(t)).$ 

(c) For the $S$-space $\mathbf Y_g,$ the subspace of the ambient space $Y$ associated with an element $t\in S\setminus R$  is $g^{-1}(Z(t)).$  Since $\beta$ is an isomorphism by assumption, $\beta^{-1}(Z(t))=V(t)$ and we have 
$$\alpha^{-1}[g^{-1}(Z(t))]=f^{-1}[\beta^{-1}(Z(t))]=f^{-1}(V(t))=X(t),$$ whence $\mathbf U_f=\mathbf U_\alpha.$
 
(d) Suppose $f$ is right minimal and $\hat f=\hat f\alpha',$ for some morphism $\alpha':\mathbf U_f\to \mathbf U_f$ in $S\spaces.$   After applying $\res^S_R$ we are in the setting of (b) where $g=f,\,\beta=1_{\mathbf V},$ and $\alpha=\res^S_R\alpha':\mathbf U\to \mathbf U.$  Since $f$  is right minimal, $f=f\alpha$ implies $\alpha$ is an isomorphism.  By (c), $\alpha'=\hat\alpha$ whence $\alpha'$ is an isomorphism according to (a).  Thus $\hat f$ is right minimal.

Suppose $\hat f$ is right minimal and $ f= f\alpha,$ for some morphism $\alpha:\mathbf U\to \mathbf U$ in $R\spaces.$  Since $f=\res^S_R1_{\mathbf V}\circ f,$ (c) says that $\hat f=\hat f\hat\alpha$ whence $\hat\alpha$ is an isomorphism.  Then $\alpha$ is an isomorphism by (a).  Therefore $f$ is right minimal.

(e) By the definition of the induction functor, $V(s)=0$ for all $s\in S\setminus F.$
\end{proof}

For the sake of completeness we present the dual statement.

\begin{prop}\label{adj1.5} Let $R$ be a filter of  a poset $S.$  For an $S$-space $\mathbf U=(U,U(s))_{s\in S}$ and an $R$-space $\mathbf V=(V,V(r))_{r\in R},$  let $g:\res_R^S\mathbf U\to\mathbf V$ be a morphism  in $R\spaces$ that is given by a $k$-linear map $g:U\to V$.  Set $\mathbf V^g=(Y,Y(s))_{s\in S}$ where  $Y=V$,  $Y(r)=V(r)$ for all $r\in R$, and   $Y(t)=g(U(t))$ for all $t\in S\setminus R.$
\begin{itemize}

\item[(a)]  $\mathbf V^g$ is an $S$-space, and the linear map $g:U\to V$ gives a morphism $\check g:\mathbf U\to\mathbf V^g$ in $S\spaces$ satisfying  $g=\res^S_R\check g.$  Moreover, $g$ is a proper morphism  (respectively, an isomorphism) if and only if so is $\check g.$

\item[(b)] Consider a commutative diagram of $R$-spaces of the form
$$ 
\begin{CD}
\res^S_R\mathbf U@>g>>\mathbf V\\
@V{\res^S_R\beta}VV@VV{\alpha}V\\
\res^S_R\mathbf X@>h>>\mathbf Z
\end{CD}
$$    
where $\mathbf X=(X,X(s))\in S\spaces,\ \mathbf Z=(Z,Z(r))\in R\spaces,$ and the morphisms $h:\res^S_R\mathbf X\to\mathbf Z,\ \alpha:\mathbf V\to\mathbf Z,$ and $\beta:\mathbf U\to\mathbf X$ are given by  $k$-linear maps $h:X\to Z,\ \alpha:V\to Z,$ and $\beta:U\to X,$ respectively.  The linear map $\alpha:V\to Z$ gives a morphism $\alpha':\mathbf V^g\to\mathbf Z^h$ in $S\spaces$ satisfying $\alpha=\res^S_R\alpha'$ and $\check h\beta=\alpha \check g.$  In particular, the functor $\res^S_R$ is left quasi full.

\item[(c)] In the setting of (b), suppose $\beta$ is an isomorphism. Then $\mathbf Z^h=\mathbf Z^\alpha$ and $\alpha'=\check\alpha:\mathbf V^g \to\mathbf Z^\alpha.$ 

\item[(d)] In the setting of (a), $g$ is left minimal if and only if so is $\check g.$

\item[(e)] If $J$ is an ideal of $S$ and $\mathbf V\cong\coind^S_J\mathbf W$, for some $\mathbf W\in J\spaces$, then  the $S$-space $\mathbf V^g=(Y,Y(s))$ of (a) satisfies $Y(s)\supset\Image g,$ for all $s\in R\setminus J$, and  $Y(s)=\Image g,$ for all $s\in S\setminus [R\cup J]$.
\end{itemize}
\end{prop}

\begin{proof}  Dual to the proof of Proposition \ref{adj1}.
\end{proof}

\section{Semisimple $S$-spaces}\label{onedim}

\begin{defn}\label{onedim.5} A nonzero $S$-space $\mathbf V$  is {\em (relatively) simple} if every nonzero proper monomorphism $\mathbf U\to\mathbf V$ in $S\spaces$ is an isomorphism.  An $S$-space is {\em (relatively) semisimple} if it is isomorphic to a direct sum of simple $S$-spaces, and we denote by $S\sem$ the full subcategory of $S\spaces$ determined by the semisimple $S$-spaces.  
\end{defn} 

\begin{rmk}\label{onedim.7} By Proposition \ref{cat10}, $\mathbf V$ is simple if and only if every nonzero proper epimorphism $\mathbf V\to\mathbf W$ in $S\spaces$ is an isomorphism. By Proposition \ref{cat3}, an $S$-space is simple if and only if  its ambient space is one-dimensional.  By Theorems \ref{projectives}(b) and  \ref{injectives}(a),  every projective and every injective $S$-space is semisimple.
\end{rmk}

It is very easy to classify the  simple $S$-spaces up to isomorphism. Let  $U$ be a finite dimensional $k$-vector space.  For each $F\in{\mathcal F}(S)$, denote by $U_F=U_{\min F}$ the $S$-space $\mathbf X=\bigl(X,X(s)\bigr)_{s\in S}$ where 
$$
X(s)=
\begin{cases}
U & \text{if $s\in F$,}\\
0 & \text{if $s\not\in F$,}
\end{cases}
$$
and, for each $I\in{\mathcal I}(S)$, denote by  $U^I=U^{\max I}$ the $S$-space $\mathbf Y=\bigl(Y,Y(s)\bigr)_{s\in S}$ where 
$$
Y(s)=
\begin{cases}
0 & \text{if $s\in I$,}\\
U & \text{if $s\not\in I$;}
\end{cases}
$$
the above notation makes sense because, in light of Propositions \ref{prel1}(a) and \ref{prel3}(a), each filter (respectively, ideal) of $S$ is uniquely determined by the antichain of its minimal (respectively, maximal) elements.

Note that for each $F\in\mathcal F(S)$ we have $U_F=U^{S\setminus F},$ and for each $I\in\mathcal I(S)$ we have $U^I=U_{S\setminus I}.$

\begin{prop}\label{onedim1}  The set $\{k_A\,|\,A\in{\mathcal A}(S)\}=\{k^A\,|\,A\in{\mathcal A}(S)\}$ is a complete set of representatives of the isomorphism classes of simple $S$-spaces.
\end{prop}  
\begin{proof}
Let ${\mathbf V}$ be a simple  $S$-space.  In view of Remark \ref{adj4}, the set  $F$ of elements of $S$ at which $\mathbf V$ is full is a filter, and the set  $I$ of elements of $S$  at which $\mathbf V$ is trivial is an ideal, of $S.$  By Remark \ref{onedim.7}, the ambient space of ${\mathbf V}$ is one-dimensional, whence $F\cup I=S$ and $F\cap I=\emptyset$.  Therefore, ${\mathbf V}\cong k_{\min F}=k^{\max I}$.  It is clear that if $A, B\in{\mathcal A}(S)$ and $A\ne B$, then  $k_A\not\cong k_B$ and $k^A\not\cong k^B$.
\end{proof}

The following statement was proved in  ~\cite{NR}.

\begin{prop}\label{width2}
Every  $S$-space is semisimple if and only if $w(S)\leq 2$.
\end{prop}
 
 We now study morphisms of semisimple $S$-spaces into arbitrary $S$-spaces.  Set 
 $$\mathcal U=\underset{A\in\hat{\mathcal A}(S)}\bigoplus k_A=\underset{A\in\check{\mathcal A}(S)}\bigoplus k^A.$$
  
 In the following two propositions we identify an element $\lambda\in k$ with the multiplication-by-$\lambda$ map $\lambda1_k: k\to k.$  

\begin{prop}\label{onedim2}  Let $\mathbf {V}=\bigl(V, V(s)\bigr)_{s\in S}$ be  an $S$-space. 
\begin{itemize} 

\item[(a)] Let  $A=a_1\wedge\dots\wedge a_m$ be in $\hat{\mathcal A}(S)$ and set $V(A)=V(a_1)\cap\dots\cap V(a_m).$ A $k$-linear map $f:k\to 
V$ gives a morphism ${k}_A\to\mathbf{V}$ of $S$-spaces if and only if $f(1)\in V(A)$.  The map $f\mapsto f(1)$ is an isomorphism $\Hom_{S\spaces}({k}_A,\mathbf{V})\cong V(A)$ of $k$-spaces functorial in $\mathbf{V}.$  We identify $\Hom_{S\spaces}(\mathcal U,\mathbf{V})$ with $\underset{C\in\hat{\mathcal A}(S)}\bigoplus V(C)$ and write the  elements of the latter as  row vectors $(v_C)_{C\in\hat{\mathcal A}(S)}$  where $v_C\in V(C).$ 

\item [(b)]   If $A, B\in \hat{\mathcal A}(S)$ then $\Hom_{S\spaces}({k}_A,{k}_B)=\begin{cases} k &\mathrm{if}\ B\leq A,\\
0 &\mathrm{otherwise}.\end{cases}$ 

\item [(c)]  We identify $\End_{{S\spaces}} \mathcal U$ with $k\hat{\mathcal A}(S)^{\mathrm{op}}$ by identifying $\Hom_{S\spaces}({k}_A,{k}_B)$    with the subspace $ke_{BA}$  of the matrix algebra $M_{\hat{\mathcal A}(S)}(k),$ for all $B\le A$ in $\hat{\mathcal A}(S).$  Then $\Hom_{S\spaces}(\mathcal U,\mathbf{V})$ is a left $k\hat{\mathcal A}(S)$-module by means of $e_{AB}\circ\bigl[(v_C)_{C\in\hat{\mathcal A}(S)}\bigr]=\bigl(\delta_{CA}v_B\bigr)_{C\in\hat{\mathcal A}(S)},$ where  $B\le A$  and $\delta_{CA}$ is the Kronecker symbol.

\item[(d)]  
$\Phi_{\hat{a}(S)}\Hom_{S\spaces}({\mathcal U},-)\cong\coind^{\hat{a}(S)}_S:S\spaces\to\hat{a}(S)\spaces$.
\end{itemize}
\end{prop}
\begin{proof} (a) Let  ${k}_A=\bigl(X,\,X(s)\bigr)_{s\in S}$.  Since a $k$-linear map $f:k\to V$ is uniquely determined by an arbitrary vector $f(1)\in V$, then $f:{k}_A\to\mathbf{V}$ is a morphism if and only if $f(1)\in V(s)$ whenever $X(s)=k,\ s\in S$.  By the definition of $k_A$,  we have $X(s)=k$ if and only if $a_i\le s$, for some $i$.  Therefore $f$ gives a morphism in $S\spaces$ if and only if $f(1)\in V(a_i),\ i=1,\dots,m$; if and only if $f(1)\in V(a_1)\cap\dots\cap V(a_m)$.

(b) Let $B=b_1\wedge\dots\wedge b_n.$ Putting ${k}_B=\bigl(V,\,V(s)\bigr)_{s\in S}$ and using (a), we see that $\Hom_{S\spaces}({k}_A, {k}_B)\ne0$ if and only if $\Hom_{S\spaces}({k}_A, {k}_B)= k$; if and only if $V(a_1)\cap\dots\cap V(a_m)=k$; if and only if $V(a_1)=\dots= V(a_m)=k$; if and only if for all $i$ there exists a $j$ such that $b_j\leq a_i$; if and only if $B\le A$ in $\hat{\mathcal A}(S)$. 

(c) Using (b) and Remark \ref{cat3.5}, we have 
$$\End_{{S\spaces}} \mathcal U=\Hom_{S\spaces}\bigl(\underset{A\in\hat{\mathcal A}(S)}\bigoplus k_A,\underset{B\in\hat{\mathcal A}(S)}\bigoplus k_B\bigr)\cong\left(\Hom_{S\spaces}({k}_A,{k}_B)\right)_{B,A\in\hat{\mathcal A}(S)}=k\hat{\mathcal A}(S)^{\mathrm{op}}.$$
Then $e_{AB}\circ(v_C)=(v_C)e_{BA}=\bigl(\delta_{CA}v_B\bigr),$ where juxtaposition indicates matrix multiplication.

(d)  In view of Proposition \ref{prel1}(c), Notation \ref{prel2}, and Definition \ref{cat6},  $\hat{\mathcal A}(S)=\hat{ a}(S)^\omega$ where $\omega=\emptyset.$ By the definition of $\Phi_{\hat{a}(S)}$,  we have $\Phi_{\hat{a}(S)}\bigl(\Hom_{S\spaces}(\mathcal U, \mathbf V)\bigr)=\bigl(X,X(B)\bigr)_{B\in\hat a(S)}$ where 
$$
X=e_{\emptyset\emptyset}\circ\bigl(\underset{C\in\hat{\mathcal A}(S)}\bigoplus V(C)\bigr)\cong V(\emptyset)=V,
$$ 
$$
X(B)=e_{\emptyset B}\circ\bigl(\underset{C\in\hat{\mathcal A}(S)}\bigoplus V(C)\bigr)\cong V(B)=V(b_1)\cap\dots\cap V(b_n),\mathrm{\ for\ all\ } B=b_1\wedge\dots\wedge b_n\mathrm{\ in\ }\hat{a}(S).
$$   Comparing these formulas with the definition of coinduction in Subsection \ref{adj}, we obtain an isomorphism $\Phi_{\hat{a}(S)}\Hom_{S\spaces}({\mathcal U},\mathbf V)\cong\coind^{\hat{a}(S)}_S\mathbf V$  functorial in $\mathbf V.$
\end{proof}

The following is a contravariant analog of the preceding statement. 

\begin{prop}\label{onedim3} Let $\mathbf {V}=\bigl(V, V(s)\bigr)_{s\in S}$ be  an $S$-space. 
\begin{itemize} 
\item[(a)] Let $B=b_1\vee\dots\vee b_n$ be in $\check{\mathcal A}(S)$ and set $V(B)=\sum^n_{j=1}V(b_j).$ A $k$-linear map $g\in\D V$ gives a morphism $g:\mathbf{V}\to{k}^B$ of $S$-spaces if and only if $g\in V(B)^{\perp},$  so there is   an isomorphism $\Hom_{S\spaces}\bigl(\mathbf{V},{k}^B\bigr)\cong V(B)^{\perp}$  of $k$-spaces functorial in $\mathbf{V}.$ We identify $\Hom_{S\spaces}(\mathbf{V},\mathcal U)$ with $\underset{C\in\check{\mathcal A}(S)}\bigoplus V(C)^\perp$ and write the elements of the latter as  column vectors $(g_C)_{C\in\check{\mathcal A}(S)}$  where $g_C\in V(C)^\perp.$ 

\item [(b)]   If $A,B\in \check{\mathcal A}(S)$ then $\Hom_{S\spaces}({k}^A,{k}^B)=\begin{cases} k &\mathrm{if}\ B\leq A,\\
0 &\mathrm{otherwise}.\end{cases}$

\item [(c)]  We identify $\End_{{S\spaces}} \mathcal U$ with  $k\check{\mathcal A}(S)^{\mathrm{op}}$ by identifying $\Hom_{S\spaces}({k}^A,{k}^B)$    with the subspace $ke_{BA}$  of the matrix algebra $M_{\check{\mathcal A}(S)}(k),$ for all $B\le A$ in $\check{\mathcal A}(S).$  Then $\Hom_{S\spaces}(\mathbf{V},\mathcal U)$ is a left $k\check{\mathcal A}(S)^{\mathrm{op}}$-module by means of $e_{BA}\bigl[(g_C)_{C\in\check{\mathcal A}(S)}\bigr]=\bigl(\delta_{CB}g_A\bigr)_{C\in\check{\mathcal A}(S)},$ where $B\le A$ and $\delta_{CB}$ is the Kronecker symbol.

\item[(d)]  
 $\Phi_{{\check{a}(S)}^{\mathrm{op}}} \Hom_{S\spaces}(-,{\mathcal U})\cong\D\ind^{\check{a}(S)}_S:S\spaces\to\check{a}(S)^{\mathrm{op}}\spaces$.
\end{itemize}
\end{prop}
\begin{proof} (a) The ambient space of $k^B$ is $k,$ and the subspace of $k$ associated to each $s\in S\setminus (B)$ is $k,$  where  $(B)$ is the ideal of $S$ generated by the antichain $B.$  Therefore, a map $g\in\D V$ is a morphism $\mathbf{V}\to{k}^B$  if and only  if $g\bigl(V(s)\bigr)=0$ for all $s\in (B)$; if and only if $g\bigl(V(b_j)\bigr)=0$ for $j=1,\dots,n$; if and only if $g\in\bigcap^n_ {j=1}\bigl(V(b_j)^{\perp}\bigr)=\bigl(\sum^n_{j=1}V(b_j)\bigr)^{\perp}$.

(b)  Let $A=a_1\vee\dots\vee a_m.$  Since any morphism $k^A\to k^B$ is  of the form $\lambda\in k$, applying (a) to ${\mathbf V}=k^A$ yields that  $\Hom_{S\spaces}({k}^A,{k}^B)\ne0$ if and only if $\lambda\in\bigl(V(B)\bigr)^{\perp}$ for some, hence for all, $\lambda\ne0$; if and only if $V(b_j)=0$ for $j=1,\dots,n$; if and only if for all $j$ there exists an $i$ such that $b_j\le a_i$; if and only if $B\le A$. 

(c) and (d) The argument is dual to the proof of parts (c) and (d) of Proposition \ref{onedim2}.

\end{proof}

\section{Projectivization}\label{pr}
We use projectivization (see \cite[Section I.2]{ARS}) to obtain equivalences of categories needed for the construction of differentiation algorithms of Section \ref{dif}.  Recall that if $U$ is an object of an additive category ${\mathfrak A}$, then $\add U$ is the full subcategory of ${\mathfrak A}$ determined by the direct summands of finite direct sums of copies of $U.$  For $X,Y\in\mathfrak A$ we denote by $\mathfrak{A}(X,Y)$ the set of morphisms from $X$ to $Y$ in $\mathfrak A.$

 The following proposition is an analog of \cite[Prop.~II.2.1]{ARS}, and the same proof works.

\begin{prop}\label{pr1}  Let $\mathfrak A$ be an additive category, let $U\in\mathfrak A$, and set $\Gamma={\mathfrak A}(U,U)$.  
\begin{itemize}
\item[(a)] The representable functor $e_U={\mathfrak A}(U, - ):{\mathfrak A}\to\Gamma^{\mathrm{op}}\Modules$ has the following properties.
\begin{itemize}
\item[(i)] $e_U:\mathfrak{A}(Z,X)\to\Hom_{\Gamma^{\mathrm{op}}}(e_U(Z),e_U(X))$ is an isomorphism for $Z\in\add U$ and $X\in\mathfrak A$.
\item[(ii)] If $X\in\add U$ then $e_U(X)\in\Gamma^{\mathrm{op}}\proj.$
\item[(iii)] $e_U|{\add U}: \add U\to\Gamma^{\mathrm{op}}\proj$ is an equivalence of categories.
\end{itemize}
\item[(b)] The contravariant representable functor $e^U={\mathfrak A}(-,U):{\mathfrak A}\to\Gamma\Modules$ has the following properties.
\begin{itemize}
\item[(i)] $e^U:\mathfrak{A}(X,Z)\to\Hom_{\Gamma}(e^U(Z),e^U(X))$ is an isomorphism for $Z\in\add U$ and $X\in\mathfrak A$.
\item[(ii)] If $X\in\add U$ then $e^U(X)\in\Gamma\proj$.
\item[(iii)] $e^U|{\add U}: \add U\to\Gamma\proj$ is a duality. 
\end{itemize}
\end{itemize}
\end{prop}

We  apply Proposition \ref{pr1} when $\mathfrak A=S\spaces$ and $U=\mathcal U=\underset{A\in\hat{\mathcal A}(S)}\bigoplus k_A=\underset{A\in\check{\mathcal A}(S)}\bigoplus k^A.$ 

\begin{prop}\label{pr2}   Let $S$ be a  poset.
\begin{itemize}
\item[(a)]  The functor $\coind^{\hat{a}(S)}_S|S\sem:S\sem\to  \hat{a}(S)\proj$ is an equivalence of categories.

\item[(b)]  An $\hat{a}(S)$-space $\mathbf W$ is projective if and only if $\res^{\hat{a}(S)}_S\mathbf W\in S\sem$ and $\mathbf W=\coind^{\hat{a}(S)}_S\res^{\hat{a}(S)}_S\mathbf W$.

\item[(c)]   The functor $\ind^{\check{a}(S)}_S|S\sem :S\sem\to \check{a}(S)\inj$ is an equivalence of categories.

\item[(d)]  An $\check{a}(S)$-space $\mathbf W$ is injective if and only if $\res^{\check{a}(S)}_S\mathbf W\in S\sem$ and $\mathbf W=\ind^{\check{a}(S)}_S\res^{\check{a}(S)}_S\mathbf W$.

\item [(e)]  If $w(S)\le2$, then $\coind^{\hat{a}(S)}_S:S\spaces\to \hat{a}(S)\proj$ and $\ind^{\check{a}(S)}_S:S\spaces\to\check{a}(S)\inj $ are equivalences of categories.

\end{itemize}
\end{prop}
\begin{proof} (a)  Since $S\sem=\add \mathcal U$, Proposition \ref{onedim2}(c) and part (iii) of Proposition \ref{pr1}(a) say that $\Hom_{S\spaces}({\mathcal U},-)|S\sem :S\sem\to  k\hat{\mathcal A}(S)\proj$ is an equivalence of categories.  In view of part (ii) of Proposition \ref{prel1}(c), Notation \ref{prel2},  and Definition \ref{cat6}, $\hat{\mathcal A}(S)=\hat{a}(S)^\omega$ where $\omega=\emptyset.$  Therefore, Theorem \ref{projectives}(d) says that  $\Phi_{\hat a(S)}|k\hat{\mathcal A}(S)\proj:k\hat{\mathcal A}(S)\proj\to\hat{a}(S)\proj $ is an equivalence of categories. Since $\coind^{\hat{a}(S)}_S\cong\Phi_{\hat{a}(S)}\Hom_{S\spaces}({\mathcal U},-)$ by Proposition \ref{onedim2}(d), the statement follows.

(b) The sufficiency follows directly from (a).  For the necessity, suppose $\mathbf W=(W,W(A))_{
A\in\hat{a}(S)}$ is in $\hat{a}(S)\proj$.  By (a), there is an isomorphism $f:\coind^{\hat{a}(S)}_S\mathbf U\to\mathbf W$ for some $\mathbf U\in S\sem$.  Applying $\res^{\hat{a}(S)}_S$ and using Proposition \ref{prel6}(b), we obtain an isomorphism $\res^{\hat{a}(S)}_S f:\mathbf U\to\res^{\hat{a}(S)}_S\mathbf W$ whence $\res^{\hat{a}(S)}_S\mathbf W\in S\sem$.  By construction, $\coind^{\hat{a}(S)}_S\mathbf U=(X,X(A))_{A\in\hat{a}(S)}$ where $X=U$ and $X(A)=U(a_1)\cap\dots\cap U(a_m)$ for all $A=a_1\wedge\dots\wedge a_m,\ m>0$.  Since $f$ is an isomorphism in $\hat{a}(S)\spaces$, the isomorphism $f:U\to W$ of $k$-spaces satisfies $f(X(A))=f(U(a_1)\cap\dots\cap U(a_m))=W(a_1)\cap\dots\cap W(a_m)=W(A)$.  Hence $\mathbf W=\coind^{\hat{a}(S)}_S\res^{\hat{a}(S)}_S\mathbf W$.

(c)  Proposition \ref{onedim3}(c) and part (iii) of Proposition \ref{pr1}(b) say that the contravariant functor $\Hom_{S\spaces}(-,{\mathcal U})|S\sem :S\sem\to  k\check{\mathcal A}(S)^{\mathrm{op}}\proj$ is a duality.  In view of part (ii) of Proposition \ref{prel3}(c), Notation \ref{prel4},  and Definition \ref{cat9}, $\check{\mathcal A}(S)=\check{a}(S)_0$ where $0=\emptyset.$  Therefore, $\check{\mathcal A}(S)^{\mathrm{op}}=\bigl({\check{a}(S)^{\mathrm{op}}}\bigr)^\omega$ where $\omega=\emptyset,$ so that Theorem \ref{projectives}(d) says that $\Phi_{\check a(S)^{\mathrm{op}}}|k\check{\mathcal A}(S)^{\mathrm{op}}\proj:k\check{\mathcal A}(S)^{\mathrm{op}}\proj\to\check{a}(S)^{\mathrm{op}}\proj $ is an equivalence of categories.  Using  the duality $\D:\check{a}(S)^{\mathrm{op}}\proj\to\check{a}(S)\inj,$ we obtain that 
$$\D\circ\Phi_{\check a(S)^{\mathrm{op}}}\circ\Hom_{S\spaces}(-,{\mathcal U})|S\sem :S\sem\to \check{a}(S)\inj$$ is an equivalence of categories.  By Proposition \ref{onedim3}(d), 
$$\ind^{\check a(S)}_S\cong\D\circ\D\circ\ind^{\check a(S)}_S\cong\D\circ\Phi_{\check a(S)^{\mathrm{op}}}\circ\Hom_{S\spaces}(-,{\mathcal U}).$$
Hence $\ind^{\check a(S)}_S|S\sem:S\sem\to\check{a}(S)\inj$ is an equivalence of categories.

(d) The argument is dual to the proof of (b).

(e)  This is an immediate consequence of (a), (c), and Proposition \ref{width2}.
\end{proof}

For a subset $R$ of a poset $S$, we denote by $(S\spaces,R\proj)$, $(S\spaces,R\inj)$, or $(S\spaces,R\sem)$ the full subcategory of $S\spaces$ determined  by the $S$-spaces $\mathbf X$ for which $\res^S_R\mathbf X$ is projective, injective, or semisimple, respectively, in $R\spaces$.  

\begin{prop}\label{pr3}  Let $T$ be a subset of a poset $S$.  For $S_{S\setminus T}=(S\setminus T)\cup\hat{a}(T)$ and $S^{S\setminus T}=(S\setminus T)\cup\check{a}(T),$ we have:
 \begin{itemize}
\item [(a)]  The functor $\coind^{S_{S\setminus T}}_S|(S\spaces,T\sem):(S\spaces,T\sem)\to  (S_{S\setminus T}\spaces,\hat{a}(T)\proj)$  is an equivalence of categories.
 \item [(b)]  The functor $\ind^{S^{S\setminus T}}_S|(S\spaces,T\sem):(S\spaces,T\sem)\to  (S^{S\setminus T}\spaces,\check{a}(T)\inj)$  is an equivalence of categories.
 \item[(c)]   If $w(T)\le2$, then  the functors $\coind^{S_{S\setminus T}}_S:S\spaces\to  (S_{S\setminus T}\spaces,\hat{a}(T)\proj)$ and  \newline$\ind^{S^{S\setminus T}}_S:S\spaces\to  (S^{S\setminus T}\spaces,\check{a}(T)\inj)$  are equivalences of categories.
\end{itemize}
\end{prop}

\begin{proof} (a) Put $\mathfrak G=(S\spaces,T\sem)$ and $\mathfrak H=(S_{S\setminus T}\spaces,\hat{a}(T)\proj).$ To check that the image of  $\coind^{S_{S\setminus T}}_S|\mathfrak G$ is contained in $\mathfrak H,$ suppose that $\mathbf V\in S\spaces$ satisfies $\res^S_T\mathbf V\in T\sem$  and set $\mathbf W=\coind^{S_{S\setminus T}}_S\mathbf V.$  We have to prove that $\res^{S_{S\setminus T}}_{\hat a(T)}\mathbf W\in \hat a(T)\proj.$ 

Using parts (b)  and (c) of Proposition \ref{prel6}, we have
$$ \res^{\hat a(T)}_T\bigl(\res^{S_{S\setminus T}}_{\hat a(T)}\mathbf W\bigr)=\res^{S_{S\setminus T}}_T\coind^{S_{S\setminus T}}_S\mathbf V=\res^S_T\mathbf V\in T\sem
$$ 
and, in view of Proposition \ref{prel6.5}(a),
$$\coind^{\hat a(T)}_T\res^{\hat a(T)}_T\bigl(\res^{S_{S\setminus T}}_{\hat a(T)}\mathbf W\bigr)=\coind^{\hat a(T)}_T\res^S_T\mathbf V=\res^{S_{S\setminus T}}_{\hat a(T)}\coind^{S_{S\setminus T}}_S\mathbf V=\res^{S_{S\setminus T}}_{\hat a(T)}\mathbf W.
$$
By Proposition \ref{pr2}(b), $ \res^{S_{S\setminus T}}_{\hat a(T)}\mathbf W$ is projective.

We also note that the image of the functor $\res^{S_{S\setminus T}}_S|\mathfrak H$ is contained in $\mathfrak G.$ Indeed, if $\mathbf W\in S_{S\setminus T}\spaces$ has the property that $\res^{S_{S\setminus T}}_{\hat a(T)}\mathbf W$ is projective, then Propositions \ref{pr2}(b) and  \ref{prel6}(c) say that 
$$\res^{\hat a(T)}_T\bigl(\res^{S_{S\setminus T}}_{\hat a(T)}\mathbf W\bigr)=\res^{S}_T\bigl(\res^{S_{S\setminus T}}_{S}\mathbf W\bigr)\in T\sem.
$$

By Proposition \ref{prel6}(b), $\bigl(\res^{S_{S\setminus T}}_S|\mathfrak H\bigr)\circ\bigl(\coind^{S_{S\setminus T}}_S|\mathfrak G\bigr)=1_{\mathfrak G},$  and we claim that 
\begin{equation}\label{pr4}
\bigl(\coind^{S_{S\setminus T}}_S|\mathfrak G\bigr)\circ\bigl(\res^{S_{S\setminus T}}_S|\mathfrak H\bigr)=1_{\mathfrak H}.
\end{equation} 

To show that equality (\ref{pr4}) holds on objects, we have to check that
\begin{equation}\label{pr5}
\coind^{S_{S\setminus T}}_S\mathbf  \res^{S_{S\setminus T}}_S\mathbf W=\mathbf W,
\end{equation} 
provided $\mathbf W\in S_{S\setminus T}\spaces$ satisfies $\res^{S_{S\setminus T}}_{\hat{a}(T)}\mathbf W\in\hat{a}(T)\proj.$   Proposition \ref{prel6}(b) says that 
$$ \res^{S_{S\setminus T}}_S\bigl( \coind^{S_{S\setminus T}}_S \res^{S_{S\setminus T}}_S\mathbf W\bigr)= \res^{S_{S\setminus T}}_S\mathbf W.
$$
Using Propositions \ref{prel6.5}(a),  \ref{prel6}(c), and \ref{pr2}(b), we also have
$$ \res^{S_{S\setminus T}}_{\hat a(T)}\bigl( \coind^{S_{S\setminus T}}_S  \res^{S_{S\setminus T}}_S\mathbf W\bigr)=\coind^{\hat a(T)}_T  \res^{S}_{T}  \res^{S_{S\setminus T}}_S\mathbf W=$$
$$\coind^{\hat a(T)}_T  \res^{\hat a(T)}_{T}  \bigl(\res^{S_{S\setminus T}}_{\hat a(T)}\mathbf W\bigr)=\res^{S_{S\setminus T}}_{\hat a(T)}\mathbf W
$$
because $\res^{S_{S\setminus T}}_{\hat{a}(T)}\mathbf W$ is projective.  Since  $S_{S\setminus T}=S\cup\hat{a}(T),$  Proposition \ref{prel6}(d) says that equality (\ref{pr5}) holds.

It is an immediate consequence of the definitions of restriction and coinduction that  equality (\ref{pr4}) also holds on morphisms.

(b)  The argument is dual to the proof of (a).

(c)  In view of Proposition \ref{width2}, this is an immediate consequence of (a) and (b).  
\end{proof}

\section{Differentiation algorithms}\label{dif}

\begin{defn}\label{prelimdif2}  Let $S$ be a poset, let $p\in S,$ and let $\mathbf{V}=(V,V(s))_{s\in S}$ be an $S$-space. We set $E^p\mathbf{V}=(X,X(s))_{s\in S}$ where $X=V(p)$ and $X(s)=V(s)\cap V(p),\ s\in S.$ Clearly, $E^p\mathbf{V}$ is an $S$-space full at $p,$ and the inclusion $\kappa_p(\mathbf V):V(p)\to V$ gives a proper monomorphism $\kappa_p(\mathbf V):E^p\mathbf{V}\to\mathbf{V}.$  By Proposition \ref{cat3}(a), $E^p\mathbf{V}$ is the only  $S$-space with the ambient space $V(p)$ for which the linear map $\kappa_p(\mathbf V)$ is a proper morphism.  For each morphism $\alpha:\mathbf{V}\to \mathbf{W}$ in $S\spaces$ given  by a $k$-linear map $\alpha:V\to W,$ where $\mathbf{W}=(W,W(s))_{s\in S},$ it is straightforward that the linear map $\alpha|V(p):V(p)\to W(p)$ gives a morphism $E^p\alpha:E^p\mathbf{V}\to E^p\mathbf{W}$ in $S\spaces.$
 
 We also set $E_p\mathbf{V}=(X,X(s))_{s\in S}$ where $X=V/V(p)$ and $X(s)=(V(s)+ V(p))/V(p),\ s\in S.$   Clearly, $E_p\mathbf{V}$ is an $S$-space trivial at $p,$ and the projection $\pi_p(\mathbf V):V\to V/V(p)$ gives a proper epimorphism $\pi_p(\mathbf V):\mathbf{V}\to E_p\mathbf{V}.$ By Proposition \ref{cat3}(b), $E_p\mathbf{V}$ is the only  $S$-space with the ambient space $V/V(p)$ for which the linear map $\pi_p(\mathbf V)$ is  a proper morphism.  For each morphism $\alpha:\mathbf{V}\to \mathbf{W}$ in $S\spaces$ given  by a $k$-linear map $\alpha:V\to W,$  the linear map $\bar\alpha:V/V(p)\to W/W(p)$ where $\bar\alpha( v+V(p))=\alpha(v)+W(p), v\in V,$ gives a morphism $E_p\alpha:E_p\mathbf{V}\to E_p\mathbf{W}$ in $S\spaces.$
\end{defn}

Recall that a morphism $\alpha:\mathbf{V}\to \mathbf{W}$ in $S\spaces$ {\em factors through} an $S$-space $\mathbf X$ if $\alpha=\beta\gamma,$ for some morphisms $\beta:\mathbf X\to\mathbf W,\ \gamma:\mathbf V\to\mathbf X.$

\begin{prop}\label{prelimdif3}  
\begin{itemize}

\item[(a)] The maps $E^p,E_p:S\spaces\to S\spaces$ are additive endofunctors.

\item[(b)] $\kappa_p:E^p\to1_{S\spaces}$ is a monomorphism, and $\pi_p:1_{S\spaces}\to E_p$ is an epimorphism, of functors.

\item[(c)] $E^p=\Ker\pi_p$  and $E_p=\Coker\kappa_p.$

\item[(d)] Let $\mathbf{V}\in S\spaces.$  The morphism $\kappa_p(\mathbf V):E^p\mathbf V\to \mathbf V$ is left minimal if and only if no nonzero direct summand of $\mathbf{V}$ is trivial at $p.$ The morphism $\pi_p(\mathbf V):\mathbf V\to E_p\mathbf V $ is right minimal if and only if no nonzero direct summand of $\mathbf{V}$ is full at $p.$

\item[(e)] Let $\alpha:\mathbf{V}\to\mathbf{W}$ be a morphism in $S\spaces.$ 

\begin{itemize}
\item[(i)] If  $\phi:E^p\mathbf{V}\to E^p\mathbf{W}$ is a morphism in $S\spaces$ for which the diagram 
$$ 
\begin{CD}
E^p\mathbf V@>\kappa_p(\mathbf V)>>\mathbf V\\
@V{\phi}VV@VV{\alpha}V\\
E^p\mathbf W@>\kappa_p(\mathbf W)>>\mathbf W
\end{CD}
$$
commutes, then $\phi=E^p\alpha.$  

\item[(ii)] If  $\psi:E_p\mathbf{V}\to E_p\mathbf{W}$ is a morphism in $S\spaces$ for which the diagram 
$$ 
\begin{CD}
\mathbf V@>\pi_p(\mathbf V)>>E_p\mathbf V\\
@V{\alpha}VV@VV{\psi}V\\
\mathbf W@>\pi_p(\mathbf W)>>E_p\mathbf W
\end{CD}
$$
commutes, then $\psi=E_p\alpha.$

\item[(iii)] $E^p\alpha=0$ if and only if $\alpha$ factors through an $S$-space trivial at $p,$ and  $E_p\alpha=0$ if and only if $\alpha$ factors through an $S$-space full at $p.$
\end{itemize}
\end{itemize}
\end{prop}
\begin{proof} The proof is routine, and we leave it to the reader.
\end{proof}

\subsection{Filters, ideals, and a dense functor}\label{prelimdif}  If a poset $S$ satisfies certain conditions, we construct a dense additive functor $S\spaces\to U\spaces,$ for some poset $U,$ and determine which morphisms of $S$-spaces the functor sends to zero.

\begin{prop}\label{prinfilt2}  Let $R$ be a filter of a poset $S$ satisfying $w(S\setminus R)\le2.$  Let $F$ be a filter of $S_R=R\cup\hat a(S\setminus R)$ that does not contain $R$ and let $p\in R\setminus F.$ For any  $F$-space ${\mathbf W}=\bigl(W,W(t)\bigr)_{t\in F},$ let 
\begin{equation}\label{prinfilt3}
f:{\mathbf P}\to\res^{S_{R}}_{\hat a(S\setminus R)}\ind_{F}^{{S_{R}}}{\mathbf W}
\end{equation} 
be a proper epimorphism given by a $k$-linear map $f:P\to W$, where ${\mathbf P}=\bigl(P,P(t)\bigr)_{t\in \hat a(S\setminus R)}$ is a projective $\hat a(S\setminus R)$-space.
\begin{itemize}

\item[(a)] There exists an $S_{R}$-space   ${\mathbf P}_f=\bigl(X,X(t)\bigr)_{t\in S_{R}}$ with $X=P$ for which the map $f:P\to W$ gives a proper epimorphism $\hat f:{\mathbf P}_f\to \ind_{F}^{{S_{R}}}{\mathbf W}$ satisfying $\res^{S_{R}}_{\hat a(S\setminus R)}\hat f=f.$  Moreover, $\Ker \hat f=E^p{\mathbf P}_f$ and $\hat f$ is a cokernel of $\kappa_p({\mathbf P}_f).$

\item[(b)] If ${\mathbf V}=\bigl(V,V(s)\bigr)_{s\in S}$ is a unique up to isomorphism $S$-space satisfying ${\mathbf P}_f\cong\coind_S^{S_{R}}{\mathbf V}$ (see Proposition \ref{pr3}(c)), then 
$${\mathbf W}\cong\res_{F}^{S_{R}}E_p\coind_S^{S_{R}}{\mathbf V}.$$

\item[(c)] If $\alpha:\mathbf{U}\to\mathbf{V}$ is a morphism in $S\spaces,$  then $\res_{F}^{S_{R}}E_p\coind_S^{S_{R}}{\alpha}=0$ if and only if $\alpha$ factors through an $S$-space full at $p.$

\item[(d)] The morphism  (\ref{prinfilt3}) is a projective cover if and only if  the $S$-space $\mathbf V$ in (b) has no nonzero direct summand full at $p$.  Hence if $\mathfrak A$ is the full subcategory of $S\spaces$ determined by the $S$-spaces with no nonzero direct summand full at $p,$  the restriction of the additive functor $\res_{F}^{S_{R}}E_p\coind_S^{S_{R}}:S\spaces \to F\spaces$ to  $\mathfrak A$ is dense.
\end{itemize}
\end{prop}

\begin{proof} (a)  By Proposition \ref{prel5.5}(a),   $\hat a(S\setminus R)=S_{R}\setminus {R}$ is an ideal of $S_{R},$  so the existence of $\hat f$ follows from Proposition \ref{adj1}(a).  Since $F$ is a filter of $S_{R}$ and $p\in R\setminus F=S_R\setminus [\hat a(S\setminus R)\cup F]$,  Proposition \ref{adj1}(e) says that $X(p)=\Ker f.$  By Proposition \ref{cat1}(c), $X(p)$ is the ambient space of $\Ker\hat f,$  and the  inclusion $\kappa_p({\mathbf P}_f):X(p)\to X$ gives a kernel of $\hat f,$ $\Ker\hat f\to\mathbf P_f,$ which is a proper monomorphism by Proposition \ref{cat3}(c).  By the remark about the uniqueness of the subspace structure on $X(p)$ made in Definition \ref{prelimdif2}, $\Ker \hat f=E^p{\mathbf P}_f.$ Since $\hat f$ is a proper epimorphism, it is a cokernel of its kernel by Proposition \ref{cat3}(d), which finishes the proof of (a).

(b)  By (a) and Proposition \ref{prelimdif3}(c), $\ind_{F}^{{S_R}}{\mathbf W}\cong E_p{\mathbf P}_f\cong E_p\coind_S^{S_{R}}{\mathbf V}.$ 
Applying the functor $\res_{F}^{{S_{R}}}$ and using Proposition \ref{prel6}(b), we get  ${\mathbf W}\cong\res_{F}^{S_{R}}E_p\coind_S^{S_{R}}{\mathbf V}.$

(c) Since $\res_{F}^{S_{R}}$ is a faithful functor by Proposition \ref{prel6}(b), $\res_{F}^{S_{R}}E_p\coind_S^{S_{R}}{\alpha}=0$ if and only if $E_p\coind_S^{S_{R}}{\alpha}=0;$ if and only if $\coind_S^{S_{R}}{\alpha}$ factors through an $S_R$-space full at $p$ according to part (iii) of Proposition \ref{prelimdif3}(e); if and only if $\alpha$ factors through an $S$-space full at $p$ in view of 
Proposition \ref{prel6}(b): indeed, if $\coind_S^{S_{R}}{\alpha}=\beta\gamma$ then ${\alpha}=\res_S^{S_{R}}\beta\circ\res_S^{S_{R}}\gamma,$ and the converse is clear.

(d) By Theorem \ref{projectives}(f), the morphism $f:{\mathbf P}\to\res^{S_{R}}_{\hat a(S\setminus R)}\ind_{F}^{{S_{R}}}{\mathbf W}$ is a projective cover if and only if it is right minimal; if and only if the morphism $\hat f:{\mathbf P}_f\to \ind_{F}^{{S_{R}}}{\mathbf W}$ is right minimal according to Proposition \ref{adj1}(d) (remember, $\hat a(S\setminus R)$ is an ideal of $S_R$); if and only if the morphism $\pi_p(\mathbf P_f):\mathbf P_f\to E_p\mathbf P_f$ is right minimal using the fact that $\hat f$ is a proper epimorphism and $\Ker\hat f=E^p\mathbf P_f$ by (a);  if and only if  $\pi_p(\coind^{S_R}_S\mathbf V):\coind^{S_R}_S\mathbf V\to E_p\coind^{S_R}_S\mathbf V$ is right minimal using  (b); if and only if no nonzero direct summand of $\coind^{S_R}_S\mathbf V$ is full at $p$ according to Proposition \ref{prelimdif3}(d); if and only if no nonzero direct summand of $\mathbf V$ is full at $p$ using Proposition \ref{prel6}(b).
\end{proof}

While Proposition \ref{prinfilt2} deals with two filters, the dual statement deals with two ideals.

\begin{prop}\label{prinideal2} Let $R$ be an ideal of a poset $S$ satisfying $w(S\setminus R)\le2.$ Let $J$ be an ideal of $S^R=R\cup\check a(S\setminus R)$ that does not contain $R$ and let $p\in R\setminus J.$ For any  $J$-space ${\mathbf W}=\bigl(W,W(t)\bigr)_{t\in J},$ let 
\begin{equation}\label{prinideal3}
g:\res_{\check{a}(S\setminus R)}^{S^{R}}\coind_{J}^{S^{R}}\mathbf W\to\mathbf I
\end{equation} 
be a proper monomorphism given by a $k$-linear map $g:W\to I$, where ${\mathbf I}=(I,I(t))_{t\in \check{a}(S\setminus R)}$ is an injective $\check{a}(S\setminus R)$-space.
\begin{itemize}

\item[(a)] There exists an $S^{R}$-space   ${\mathbf I}^g=\bigl(Y,Y(t)\bigr)_{t\in S^{R}}$ with $Y=I$ for which the map $g:W\to I$ gives a proper monomorphism $\check g:\coind_{J}^{S^{R}}{\mathbf W}\to {\mathbf I}^g$ satisfying $\res^{S^{R}}_{\check a(S\setminus R)}\check g=g.$  Moreover, $\Coker\check g=E_p\mathbf I^g$ and $\check g$ is a kernel of $\pi_p(\mathbf I^g)$.

\item[(b)] If ${\mathbf V}=\bigl(V,V(s)\bigr)_{s\in S}$ is a unique up to isomorphism $S$-space satisfying ${\mathbf I}^g\cong\ind_S^{S^{R}}{\mathbf V}$ (see Proposition \ref{pr3}(c)), then 
$${\mathbf W}\cong\res_{J}^{S^{R}}E^p\ind_S^{S^{R}}{\mathbf V}.$$

\item[(c)] If $\alpha:\mathbf U\to\mathbf V$ is a morphism in $S\spaces,$ then $\res_{J}^{S^{R}}E^p\ind_S^{S^{R}}{\alpha}=0$ if and only if $\alpha$  factors through an $S$-space trivial at $p.$

\item[(d)] The morphism  (\ref{prinideal3}) is an injective envelope if and only if the $S$-space $\mathbf V$ in (b) has no nonzero direct summand trivial at $p$. Hence if $\mathfrak B$ is the full subcategory of $S\spaces$ determined by the $S$-spaces with no nonzero direct summand trivial at $p,$  the restriction of the additive functor $\res_{J}^{S^{R}}E^p\ind_S^{S^{R}}:S\spaces \to J\spaces$ to  $\mathfrak B$ is dense.
\end{itemize}
\end{prop}

\begin{proof} The proof is dual to that of Proposition \ref{prinfilt2}.
\end{proof}

In the following subsections we will apply Propositions \ref{prinfilt2} and \ref{prinideal2} by making specific choices for the indicated filters and ideals.

 \subsection{Differentiation with respect to a principal filter}\label{prinfilt}
\begin{defn}\label{prinfilt1}  For any $p\in S$, the poset $S_{\langle p\rangle}=\langle p\rangle\cup\hat a\bigl(S\setminus \langle p\rangle\bigr)$ is a subposet of $\hat a(S),$ and Propositions  \ref{prel2.5}(a) and \ref{prel.5}(b) imply that $\hat a((p))$ is the principal ideal of $S_{\langle p\rangle}$ generated by $p.$  Hence 
$$S_p=S_{\langle p\rangle}\setminus\hat a((p))=\bigl(\langle p\rangle\setminus \{p\}\bigr)\cup\hat a\bigl(S\setminus[\langle p\rangle\cup(p)]\bigr)$$ 
is a filter of $S_{\langle p\rangle}$ satisfying $\{p\}=\langle p\rangle\setminus S_p.$ 

For the rest of this subsection we assume that  $w(S\setminus \langle p\rangle)\le2.$  Then Proposition \ref{prinfilt2} applies to the filters $R={\langle p\rangle}$ of $S$  and $F=S_p$ of $S_{\langle p\rangle},$ and we  say that the functor $\res_{S_p}^{S_{\langle p\rangle}}E_p\coind_S^{S_{\langle p\rangle}}:S\spaces\to S_p\spaces$ suggested by Proposition \ref{prinfilt2}(b) is the {\em differentiation functor}, and $S_p$ is the {\it derived poset} of $S,$ with respect to the principal filter $\langle p\rangle$.  
\end{defn}

 Recall that the category $S\spaces$ is a {\em $k$-category}, i.e., $\Hom_{S\spaces}(\mathbf U,\mathbf V)$ is a $k$-vector space, for all $\mathbf U,\mathbf V,$ and the composition of morphisms is bilinear (see ~\cite[Section II.1]{ARS}).  Denote by $\mathfrak F(\mathbf U,\mathbf V)$ the subset of $\Hom_{S\spaces}(\mathbf U,\mathbf V)$  consisting of all morphisms that factor through an $S$-space full at $p.$  Then $\mathfrak F$ is a {\em two-sided ideal} in the category $S\spaces$ (a {\em relation} on $S\spaces$ according to  ~\cite{ARS}), i.e.,  $\mathfrak F(\mathbf U,\mathbf V)$ is a $k$-subspace of $\Hom_{S\spaces}(\mathbf U,\mathbf V),$ for all $\mathbf U,\mathbf V,$ and whenever $g\in\mathfrak F(\mathbf U,\mathbf V),\, f\in\Hom_{S\spaces}(\mathbf X,\mathbf U),\, h\in\Hom_{S\spaces}(\mathbf V,\mathbf W),$ we have $hgf\in\mathfrak F(\mathbf X,\mathbf W).$ One defines  $S\spaces/\mathfrak F,$ the {\em quotient category} ({\em factor category} according to  ~\cite{ARS}) of $S\spaces$ modulo the ideal $\mathfrak F,$ as follows.   The objects of $S\spaces/\mathfrak F$ are the same as those of $S\spaces.$  The morphisms from $\mathbf U$ to $\mathbf V$ are the elements of the quotient (factor) space $\Hom_{S\spaces}(\mathbf U,\mathbf V)/\mathfrak F(\mathbf U,\mathbf V),$  and the composition in $S\spaces/\mathfrak F$ is defined for $\mathbf U,\mathbf V,\mathbf W$ in $S\spaces/\mathfrak F$ by $\bigl(h+\mathfrak F(\mathbf V,\mathbf W)\bigr)\bigl(g+\mathfrak F(\mathbf U,\mathbf V)\bigr)=\bigl(hg+\mathfrak F(\mathbf U,\mathbf W)\bigr),$ for all $g\in\Hom_{S\spaces}(\mathbf U,\mathbf V)$ and $h\in\Hom_{S\spaces}(\mathbf V,\mathbf W).$ 

\begin{thm}\label{prinfilt4}
\begin{itemize}

\item[(a)] Denote by $\mathfrak A$ the full subcategory of $S\spaces$ determined by the $S$-spaces with no nonzero direct summand full at $p.$  The restriction  of the functor $\res_{S_p}^{S_{\langle p\rangle}}E_p\coind_S^{S_{\langle p\rangle}}:S\spaces\to S_p\spaces$ to the subcategory $\mathfrak A$ is a representation equivalence of categories $\mathfrak A\to S_p\spaces.$

\item[(b)]  The functor $\res_{S_p}^{S_{\langle p\rangle}}E_p\coind_S^{S_{\langle p\rangle}}$ induces an equivalence of categories $S\spaces/\mathfrak F\cong  S_p\spaces.$

\item[(c)]  For each $S$-space $\mathbf V=\bigl(V,V(s)\bigr)_{s\in S},$ we have $\res_{S_p}^{S_{\langle p\rangle}}E_p\coind_S^{S_{\langle p\rangle}}\mathbf V=\bigl(X,X(t)\bigr)_{t\in S_p}$ where $X=V/V(p),\ X(t)=\bigl(V(t)+V(p)\bigr)/V(p)$ if $t\in S\setminus (p),$ and $X(a\wedge b)=(V(a)\cap V(b)+V(p))/V(p)$ if $a,b\in S\setminus [\langle p\rangle\cup(p)]$ and $\{p,a,b\}$ is an antichain.  If $\mathbf U=\bigl(U,U(s)\bigr)_{s\in S}$ is an $S$-space and $f:\mathbf U\to\mathbf V$ is a mophism given by a $k$-linear map $f:U\to V,$ then the morphism $\res_{S_p}^{S_{\langle p\rangle}}E_p\coind_S^{S_{\langle p\rangle}}f$ is given by the $k$-linear map $\bar f:U/U(p)\to V/V(p)$ where $\bar f(u+U(p))=f(u)+V(p),u\in U.$
\end{itemize}
\end{thm}

\begin{proof}  (a)  The functor is dense according to Proposition  \ref{prinfilt2}(d), so it remains to show that the functor is full and reflects isomorphisms.  

As noted in Definition \ref{prinfilt1},  $S_p$ is a filter  of $S_{\langle p\rangle},$ and $S_{\langle p\rangle}\setminus S_p$ is the principal ideal of $S_{\langle p\rangle}$ generated by $p.$ By Proposition \ref{prel6}(e),  $\res^{S_{\langle p\rangle}}_{S_p}|\mathfrak C:\mathfrak C\to S_p\spaces$ is an equivalence of categories, where $\mathfrak C$ is the full subcategory of $S_{\langle p\rangle}\spaces$ determined by the $S_{\langle p\rangle}$-spaces trivial at $p.$  By  Definition \ref{prelimdif2}, the image of the functor $E_p:S_{\langle p\rangle}\spaces\to S_{\langle p\rangle}\spaces$ is contained in $\mathfrak C.$ Hence it suffices to show that the functor $E_p\coind_S^{S_{\langle p\rangle}}:\mathfrak A\to {S_{\langle p\rangle}}\spaces$  is full and reflects isomorphisms. Multiplying the epimorphism of functors $\pi_p:1_{S_{\langle p\rangle}\spaces}\to E_p$ by the functor $\coind_S^{S_{\langle p\rangle}}$ on the right, we obtain an epimorphism of functors $\pi_p\coind_S^{S_{\langle p\rangle}}:\coind_S^{S_{\langle p\rangle}}\to E_p\coind_S^{S_{\langle p\rangle}}.$

Since $\langle p\rangle$ is a filter of $S_{\langle p\rangle}$  by Proposition \ref{prel5.5}(a), then $\hat a(S\setminus\langle p\rangle)$  is an ideal, so Proposition  \ref{adj1} applies.

We show that the functor $E_p\coind_S^{S_{\langle p\rangle}}:S\spaces\to  {S_{\langle p\rangle}}\spaces$ is full. Let $\mathbf {V}=\bigl(V,V(s)\bigr)_{s\in S},\ \mathbf {Z}=\bigl(Z,Z(s)\bigr)_{s\in S}$ be $S$-spaces and let $\beta:E_p\coind_S^{S_{\langle p\rangle}}\mathbf V\to E_p\coind_S^{S_{\langle p\rangle}}\mathbf Z$ be a morphism in $S_{\langle p\rangle}\spaces.$    We obtain the diagram 

$$
\begin{CD}
\coind_S^{S_{\langle p\rangle}}\mathbf V@>\pi_p\bigl(\coind_S^{S_{\langle p\rangle}}\mathbf V\bigr)>>E_p\coind_S^{S_{\langle p\rangle}}\mathbf V\\
@.@VV{\beta}V\\
\coind_S^{S_{\langle p\rangle}}\mathbf Z@>\pi_p\bigl(\coind_S^{S_{\langle p\rangle}}\mathbf Z\bigr)>>E_p\coind_S^{S_{\langle p\rangle}}\mathbf Z
\end{CD}
$$
in $S_{\langle p\rangle}\spaces.$ As noted in Definition \ref{prelimdif2}, the horizontal arrows are proper epimorphisms.  
 
Applying the functor $\res^{S_{\langle p\rangle}}_{\hat a(S\setminus\langle p\rangle)},$ we obtain the commutative diagram
\begin{equation}\label{prinfilt5} 
\begin{CD}
\res^{S_{\langle p\rangle}}_{\hat a(S\setminus\langle p\rangle)}\coind_S^{S_{\langle p\rangle}}\mathbf V@>f>>\res^{S_{\langle p\rangle}}_{\hat a(S\setminus\langle p\rangle)}E_p\coind_S^{S_{\langle p\rangle}}\mathbf V\\
@V\alpha VV@VV{\res^{S_{\langle p\rangle}}_{\hat a(S\setminus\langle p\rangle)}\beta}V\\
\res^{S_{\langle p\rangle}}_{\hat a(S\setminus\langle p\rangle)}\coind_S^{S_{\langle p\rangle}}\mathbf Z@>g>>\res^{S_{\langle p\rangle}}_{\hat a(S\setminus\langle p\rangle)}E_p\coind_S^{S_{\langle p\rangle}}\mathbf Z
\end{CD}
\end{equation}
in $\hat a(S\setminus\langle p\rangle)\spaces,$  where $f=\res^{S_{\langle p\rangle}}_{\hat a(S\setminus\langle p\rangle)}\pi_p\bigl(\coind_S^{S_{\langle p\rangle}}\mathbf V\bigr),\ g=\res^{S_{\langle p\rangle}}_{\hat a(S\setminus\langle p\rangle)}\pi_p\bigl(\coind_S^{S_{\langle p\rangle}}\mathbf Z\bigr)$ and we denote by the same letters the $k$-linear maps that give these morphisms: $f:V\to V/V(p),\ g: Z\to Z/Z(p).$ The morphism $\alpha$ making the diagram commute exists because $f$ and $g$ are proper epimorphisms by Proposition \ref{cat1}(b) and Remark \ref{adj2}, and  Proposition \ref{pr3}(c) says that the domains of $f$ and $g$ are projective because $w(S\setminus\langle p\rangle)\le2.$

The subspaces $V(p)\subset V$ and $Z(p)\subset Z$ are the kernels of the $k$-linear maps $f:V\to V/V(p)$ and $g:Z\to Z/Z(p),$ respectively,  whence $V(t)=f^{-1}\bigl[\bigl(V(t)+V(p)\bigr)/V(p)\bigr]$ and $Z(t)=f^{-1}\bigl[\bigl(Z(t)+Z(p)\bigr)/Z(p)\bigr]$ for all $t\ge p,\ t\in S_{\langle p\rangle}.$  In the notation of Proposition  \ref{adj1}(a), we have $\hat f=\pi_p\bigl(\coind_S^{S_{\langle p\rangle}}\mathbf V\bigr)$ and $\hat g=\pi_p\bigl(\coind_S^{S_{\langle p\rangle}}\mathbf Z\bigr),$ and Proposition \ref{adj1}(b)  gives the  commutative diagram

\begin{equation}\label{prinfilt4.5} 
\begin{CD}
\coind_S^{S_{\langle p\rangle}}\mathbf V@>\hat f>>E_p\coind_S^{S_{\langle p\rangle}}\mathbf V\\
@V{\alpha'}VV@VV{\beta}V\\
\coind_S^{S_{\langle p\rangle}}\mathbf Z@>\hat g>>E_p\coind_S^{S_{\langle p\rangle}}\mathbf Z
\end{CD}
\end{equation}
in $S_{\langle p\rangle}\spaces,$ where $\res^{S_{\langle p\rangle}}_{\hat a(S\setminus\langle p\rangle)}\alpha'=\alpha.$
By Proposition \ref{prelimdif3}(e) we have $\beta=E_p\alpha',$  and Proposition \ref{prel6}(b) says that $\alpha'=\coind_S^{S_{\langle p\rangle}}\gamma,$ for some morphism $\gamma$ in $S\spaces.$  Thus $\beta=E_p\coind_S^{S_{\langle p\rangle}}\gamma,$ which proves that the functor $E_p\coind_S^{S_{\langle p\rangle}}$ is full.

To show the functor $E_p\coind_S^{S_{\langle p\rangle}}:\mathfrak A\to {S_{\langle p\rangle}}\spaces$ reflects isomorphisms, let $\mathbf V,\mathbf Z\in\mathfrak A$ and let $\gamma:\mathbf V\to\mathbf Z$ be a morphism in $S\spaces$ for which  $\beta=E_p\coind_S^{S_{\langle p\rangle}}\gamma:E_p\coind_S^{S_{\langle p\rangle}}\mathbf V\to E_p\coind_S^{S_{\langle p\rangle}}\mathbf Z$ is an isomorphism in $S_{\langle p\rangle}\spaces.$ Setting $\alpha'=\coind_S^{S_{\langle p\rangle}}\gamma:\coind_S^{S_{\langle p\rangle}}\mathbf V\to\coind_S^{S_{\langle p\rangle}}\mathbf Z,$ we see that $\alpha'$ and $\beta$ just defined make the diagram (\ref{prinfilt4.5}) commute because $\pi_p\coind_S^{S_{\langle p\rangle}}:\coind_S^{S_{\langle p\rangle}}\to E_p\coind_S^{S_{\langle p\rangle}}$ is a natural transformation.  Setting $\alpha=\res^{S_{\langle p\rangle}}_{\hat a(S\setminus\langle p\rangle)}\alpha',$ we get that the diagram (\ref{prinfilt5}) also commutes. Since $\beta$ is an isomorphism, Proposition \ref{adj1}(c) says that $\alpha'=\hat\alpha.$

Since coinduction is a fully  faithful additive functor by Proposition \ref{prel6}(b), each direct summand of $\coind^{S_{\langle p\rangle}}_S\mathbf V$ is isomorphic to $\coind^{S_{\langle p\rangle}}_S\mathbf X,$ where $\mathbf X$ is a direct summand of  $\mathbf V.$  Since  $\mathbf V,\mathbf Z\in\mathfrak A,$ no nonzero direct summand of  $\coind^{S_{\langle p\rangle}}_S\mathbf V$ or $\coind^{S_{\langle p\rangle}}_S\mathbf Z$ is full at $p.$  By Proposition \ref{prelimdif3}(d), $\hat f$ and $\hat g$ are right minimal morphisms.  By Proposition \ref{adj1}(d), $f$ and $g$ are right minimal morphisms, and we already noted that  they are proper epimorphisms. Hence, they are projective covers by Theorem \ref{projectives}(f).  Since $\beta$ is an isomorphism, so is $\alpha,$  and Proposition \ref{adj1}(a) says that $\hat\alpha=\alpha'$ is an isomorphism.  Since coinduction reflects isomorphisms by Proposition  \ref{prel6}(b), $\gamma$ is an isomorphism.  We have proved that the functor $E_p\coind_S^{S_{\langle p\rangle}}$ restricted to $\mathfrak A$ reflects isomorphisms.

(b)  This is a direct consequence of (a) and Proposition \ref{prinfilt2}(c).

(c)  In view of the way the posets $S_{\langle p\rangle}$ and $S_p$ are constructed, this follows immediately from the definition of coinduction and Definition \ref{prelimdif2}.
\end{proof}

Recall ~\cite{GR} that the poset $S$ is   \emph{representation-finite} if the category $kS\modules$ has only finitely many isomorphism classes of indecomposable modules, and $S$ is {\em finitely represented} if the cardinality $\nu(S)$ of the set of isomorphism classes of indecomposable $S$-spaces is finite. 

A characterization of finitely represented posets is given in ~\cite{Kl}, whereas a characterization and description of representation-finite posets in given in ~\cite{Lou} and ~\cite{ZaSh}.

We have the following consequence of the preceding theorem. 

\begin{cor}\label{corfilt} If $p\in S$ satisfies  $w(S\setminus\langle p\rangle)\leq 2,$ then
$$
\nu(S)=\nu(S_p)+\bigl| a\bigl(S\setminus\langle p\rangle\bigr)\bigr|+1.
$$
In particular, $S$ is finitely represented  if and only if so is $S_p.$
\end{cor}

\begin{proof}
Since a representation equivalence of categories establishes a bijection between isomorphism classes of indecomposable objects, Theorem \ref{prinfilt4} implies that $\nu(S)$ equals $\nu(S_p)$ plus the cardinality of the set of isomorphism classes of indecomposable $S$-spaces full at $p$. Since $S\setminus\langle p\rangle$ is an ideal of $S,$ Proposition \ref{prel6}(f) says that the latter cardinality is  $\nu(S\setminus\langle p\rangle)$. By Proposition \ref{width2}, every indecomposable $(S\setminus\langle p\rangle)$-space is simple. In view of Notation \ref{prel2} and using the bijection between antichains and isomorphism classes of simple $(S\setminus\langle p\rangle)$-spaces established by Proposition \ref{onedim1}, we get $\nu(S\setminus\langle p\rangle)=\bigl|\mathcal A(S\setminus\langle p\rangle)\bigr|=\bigl| a(S\setminus\langle p\rangle)\bigr|+1.$
\end{proof}

Recall that the {\em Hasse diagram}  of a poset $S$ is the quiver with the set of vertices $S$ in which there is a single arrow $a\to b$ if and only if $a<b$ and no element  $c\in S $ satisfies $a< c< b;$  there are no other arrows in the Hasse diagram.

\begin{ex}\label{prinfilt6} To illustrate the differentiation with respect to a principal filter, consider the poset $S$ given by the following Hasse diagram.
$$
S:\quad 
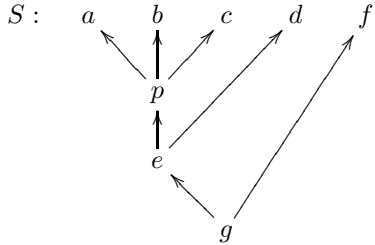
\xymatrix@=1.5em{
    a & b & c & d &f\\
    & p\ar[ul]\ar[u]\ar[ur]\\
    & e\ar[u]\ar[uurr] & \\
    & & g\ar[ul]\ar[uuurr]}
$$   
Then the Hasse diagrams of $S_p$ is as follows.
$$
S_p:\quad
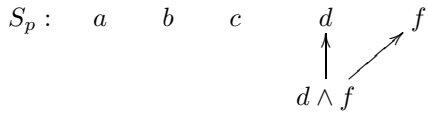
\xymatrix@=1.5em{
    a & b & c & d & f\\
    & & & d\wedge f\ar[u]\ar[ur]\\} 
$$
Note that three differentiations with respect to minimal elements (first $g$, then $e$, and finally $p$) are needed to reduce $S$ to the same poset as our $S_p$ above.
\end{ex} 

\subsection{Differentiation with respect to a principal ideal}\label{prinideal}
The constructions that follow are dual to the ones of the previous subsection.

\begin{defn}\label{prinideal1}  For any $p\in S$, the poset $S^{(p)}=(p)\cup\check{a}\bigl(S\setminus(p)\bigr)$ is a subposet of $\check a(S),$ and Propositions \ref{prel2.5}(b) and \ref{prel.5}(a) imply that $\check a(\langle p\rangle)$ is the principal filter of $S^{(p)}$ generated by $p.$  Hence
$$
S^p=S^{(p)}\setminus \check a(\langle p\rangle)=\bigl((p)\setminus\{p\}\bigr)\cup\check{a}\bigl(S\setminus[(p)\cup\langle p\rangle]\bigr)
$$ 
is an ideal of $S^{(p)}$ satisfying $\{p\}=(p)\setminus S^p.$

For the rest of this subsection we assume that  $w(S\setminus(p))\le2.$  Then Proposition \ref{prinideal2} applies to the ideals $R=(p)$ of $S$  and $J=S^p$ of $S^{(p)},$ and we  say that the functor $\res_{S^p}^{S^{(p)}}E^p\ind_S^{S^{(p)}}:S\spaces\to S^p\spaces$ suggested by Proposition \ref{prinideal2}(b) is the {\em differentiation functor}, and $S^p$ is the {\em derived poset} of $S,$ with respect to the principal ideal $(p).$
\end{defn}

For all $\mathbf U,\mathbf V\in S\spaces,$ denote by $\mathfrak T(\mathbf U,\mathbf V)$ the subset of $\Hom_{S\spaces}(\mathbf U,\mathbf V)$  consisting of all morphisms that factor through an $S$-space trivial at $p.$  Then $\mathfrak T$ is a two-sided ideal (relation) in $S\spaces.$ 

\begin{thm}\label{prinideal4}  
\begin{itemize}

\item[(a)] Denote by $\mathfrak B$ the full subcategory of $S\spaces$ determined by the $S$-spaces with no nonzero direct summand trivial at $p.$  The restriction  of the functor $\res_{S^p}^{S^{(p)}}E^p\ind_S^{S^{(p)}}:S\spaces\to S^p\spaces$ to the subcategory $\mathfrak B$ is a representation equivalence of categories $\mathfrak B\to S^p\spaces.$

\item[(b)]  The functor $\res_{S^p}^{S^{(p)}}E^p\ind_S^{S^{(p)}}$ induces an equivalence of categories $S\spaces/\mathfrak T\cong  S^p\spaces.$

\item[(c)]  For each $S$-space $\mathbf V=\bigl(V,V(s)\bigr)_{s\in S},$ we have $\res_{S^p}^{S^{(p)}}E^p\ind_S^{S^{(p)}}\mathbf V=\bigl(X,X(t)\bigr)_{t\in S^p}$ where $X=V(p),\ X(t)=V(t)\cap V(p)$ if $t\in S\setminus \langle p\rangle,$ and $X(a\vee b)=\bigl(V(a)+ V(b)\bigr)\cap V(p)$ if $a,b\in S\setminus [(p)\cup\langle p\rangle]$ and $\{p,a,b\}$ is an antichain.  If $\mathbf U=\bigl(U,U(s)\bigr)_{s\in S}$ is an $S$-space and $f:\mathbf U\to\mathbf V$ is a mophism given by a $k$-linear map $f:U\to V,$ then the morphism $\res_{S_p}^{S_{\langle p\rangle}}E_p\coind_S^{S_{\langle p\rangle}}f$ is given by the $k$-linear map $f|U(p):U(p)\to V(p).$
\end{itemize}
\end{thm}

\begin{proof}  The proof is dual to that of Theorem \ref{prinfilt4}.
\end{proof}

\begin{cor} If $p\in S$ satisfies $w(S\setminus(p))\leq 2,$ then
$$
\nu(S)=\nu(S^p)+\bigl| a\bigl(S\setminus(p)\bigr)\bigr|+1.
$$
In particular, $S$ is  finitely represented if and only if so is $S^p.$
\end{cor}

\begin{proof}
Dual to the proof of Corollary \ref{corfilt}.
\end{proof}
                     
\subsection{Differentiation and Duality}\label{connections}
We show that the duality $\D$  commutes with the functors $E^p$ and $E_p.$  By Proposition \ref{prel7},  the duality commutes with restriction, induction, and coinduction. Hence, it  commutes with the differentiation functors with respect to a principal filter and to a principal ideal considered in the subsections \ref{prinfilt} and \ref{prinideal}, respectively. 

\begin{lem}\label{F and D}
Let $p\in S$. Then $\D E_p\cong E^p\D$, i.\,e., the following diagram commutes up to isomorphism.
$$
\begin{CD}S\spaces@>E_p>>S\spaces\\
@V \D VV@VV \D V\\
S\op\spaces@>E^p>>S\op\spaces
\end{CD}
$$
\end{lem}

\begin{proof}
Let $\mathbf V$ be an $S$-space. Then $E_p\mathbf V$ is the $S$-space $\bigl(X,X(s)\bigr)$ with $X=V/V(p)$ and $X(s)=\bigl(V(s)+V(p)\bigr)/V(p)$, so that $\D E_p\mathbf V$ is the $S^{\mathrm{op}}$-space $\bigl(Y,Y(s)\bigr)$ with $Y=\D \bigl(V/V(p)\bigr)$ and 
$$Y(s)=\bigl[\bigl(V(s)+V(p)\bigr)/V(p)\bigr]^\perp=\{f\in\D \bigl(V/V(p)\bigr)\,|\,f\bigl[\bigl(V(s)+V(p)\bigr)/V(p)\bigr]=0\}.$$ 
On the other hand, $\D\mathbf V$ is the $S^{\mathrm{op}}$-space $(X',X'(s))$ with $X'=\D V$ and $X'(s)=V(s)^\perp$, so that $E^p\D\mathbf V$ is the $S^{\mathrm{op}}$-space $(Y',Y'(s))$ with $Y'=X'(p)=V(p)^\perp$ and $$Y'(s)=X'(s)\cap X'(p)=V(s)^\perp\cap V(p)^\perp=(V(s)+V(p))^\perp.$$ 
Note that the $k$-linear map $\varphi(\mathbf V):\D \bigl(V/V(p)\bigr)\to V(p)^\perp$ given by $[\varphi(\mathbf V)](f)=f\circ \pi_p(\mathbf V)$, where $\pi_p(\mathbf V):V\to V/V(p)$ is the natural projection, is an isomorphism of $k$-spaces. Also observe that $[\varphi(\mathbf V)]\bigl(Y(s)\bigr)=Y'(s)$ for $s\in S$, whence $\varphi(\mathbf V):\D E_p\mathbf V\to E^p\D\mathbf V$ is an isomorphism of $S\op$-spaces. It easy to check that the family $\varphi=\bigl(\varphi(\mathbf V)\bigr)_{\mathbf V\in S\spaces}$ is a natural transformation $\varphi:\D E_p\to E^p\D$. 
 \end{proof}
The following statement  imposes no restrictions on the element $p$ and thus extends ~\cite[Corollary 7.10, p. 85]{S}.
\begin{prop}\label{minmax}
Let $p\in S$. Then the following diagram commutes up to isomorphism.
$$
\begin{CD}S\spaces@>\res_{S_p}^{S_{\langle p\rangle}}E_p\coind_S^{S_{\langle p\rangle}}>>S_p\spaces\\
@V \D VV@VV \D V\\
S\op\spaces@>\res_{(S\op)^p}^{(S\op)^{(p)}}E^p\ind_{S\op}^{(S\op)^{(p)}}>>(S\op)^p\spaces
\end{CD}
$$
\end{prop}

\begin{proof}
This follows from Proposition \ref{prel7} and Lemma \ref{F and D}.
\end{proof}

\end{document}